\documentclass[USenglish, a4paper, 10pt]{article}

\input{TeX/Preamble.tex}

	\title{%
		On the natural domain of Bregman operators\thanks{%
			The work of A. Themelis was supported by the JSPS KAKENHI grants number JP21K17710 and JP24K20737.
			Z. Wang acknowledges the support of Xianfu Wang's NSERC Discovery Grants and the Mitacs Gloablink Award.
		}%
	}
	\author{%
		Andreas Themelis\thanks{%
			Faculty of Information Science and Electrical Engineering, Kyushu University.
			744 Motooka, Nishi-ku, Fukuoka 819-0395, Japan.
			{\it E-mail:} {\sf andreas.themelis@ees.kyushu-u.ac.jp}%
		}\and
		Ziyuan Wang\thanks{%
			Department of Mathematics, University of Vienna.
			Oskar-Morgenstern-Platz 1, 1090 Vienna, Austria.
			{\it E-mail:} {\sf ziyuanw96@univie.ac.at}%
		}%
	}
	\date{}

\begin{document}

	\maketitle

	\begin{abstract}
		The Bregman proximal mapping and Bregman-Moreau envelope are traditionally studied for functions defined on the entire space \(\R^n\), even though these constructions depend only on the values of the function within (the interior of) the domain of the distance-generating function (dgf).
While this convention is largely harmless in the convex setting, it leads to substantial limitations in the nonconvex case, as it fails to embrace important classes of functions such as relatively weakly convex ones.
In this work, we revisit foundational aspects of Bregman analysis by adopting a domain-aware perspective: we define functions on the natural domain induced by the dgf and impose properties only relative to this set.
This framework not only generalizes existing results but also rectifies and simplifies their statements and proofs.
Several examples illustrate both the necessity of our assumptions and the advantages of this refined approach.

		\amssubj{%
			49J52, 
			49J53, 
			49M27, 
			47H04, 
			26B25.
		}%
	\end{abstract}

	\tableofcontents

	\section{Introduction}
		The \emph{proximal mapping} and the \emph{Moreau envelope} are fundamental tools in variational analysis and optimization.
Given a proper, lower semicontinuous function \(\func{f}{\R^n}{\Rinf\coloneqq\R\cup\set{\pm\infty}}\), the \emph{proximal mapping} with parameter \(\lambda > 0\) is the set-valued operator \(\ffunc{\Eprox_{\lambda f}}{\R^n}{\R^n}\) defined as
\begin{equation}\label{eq:Eprox}
	\Eprox_{\lambda f}(y)
\coloneqq
	\argmin_{x \in \R^n} \set{ f(x) + \tfrac{1}{2\lambda} \norm{x - y}^2},
\end{equation}
while the \emph{Moreau envelope} is the associated value function \(\func{\Eenv_\lambda f}{\R^n}{\Rinf}\):
\begin{equation}\label{eq:Eenv}
	\Eenv_\lambda f(y)
\coloneqq
	\inf_{x \in \R^n} \set{ f(x) + \tfrac{1}{2\lambda} \norm{x - y}^2}.
\end{equation}
The proximal mapping is a building block of many optimization algorithms, while the Moreau envelope \(\Eenv_\lambda f\) serves as a regularized approximation of \(f\) that plays a key role in the theoretical analysis of their convergence.

Replacing the quadratic term with a \emph{Bregman distance} \(\func{\D}{\R^n\times\R^n}{\Rinf}\) induced by a proper, lsc, convex function \(\func{\kernel}{\R^n}{\Rinf}\), namely
\begin{equation}
	\D(x,y)
	=
	\begin{cases}
		\kernel(x)-\kernel(y)-\innprod{\nabla\kernel(y)}{x-y} & \text{if } y\in\interior\dom\kernel
		\\
		\infty & \text{otherwise,}
	\end{cases}
\end{equation}
leads to the \emph{Bregman} counterparts of \eqref{eq:Eprox} and \eqref{eq:Eenv} that will be the focus of this paper.
Since the Bregman distance is typically asymmetric, the minimizations with respect to the first or second argument may differ, which gives rise to a \emph{left} and a \emph{right} Bregman proximal operator, denoted
\begin{subequations}\label{subeq:prox}
	\begin{align}
		\prox_{\lambda f}(\bar y)
	\coloneqq{} &
		\argmin\set{f+\tfrac{1}{\lambda}\D({}\cdot{},\bar y)}
	\shortintertext{and}
		\prox*_{\lambda g}(\bar x)
	\coloneqq{} &
		\argmin\set{g+\tfrac{1}{\lambda}\D(\bar x,{}\cdot{})},
	\end{align}
\end{subequations}
as well as the corresponding \emph{left} and \emph{right} envelopes
\begin{subequations}\label{subeq:env}
	\begin{align}
		\env_\lambda f(\bar y)
	\coloneqq{} &
		\inf\set{f+\tfrac{1}{\lambda}\D({}\cdot{},\bar y)}
	\shortintertext{and}
		\env*_\lambda g(\bar x)
	\coloneqq{} &
		\inf\set{g+\tfrac{1}{\lambda}\D(\bar x,{}\cdot{})},
	\end{align}
\end{subequations}
all defined for extended-real-valued functions \(f\) and \(g\).
These generalizations are of central importance in optimization and variational analysis; see for instance
the Chebyshev problem in the sense of the Bregman distance \cite{bauschke2009bregmanC},
various Bregman iterative methods \cite{bauschke2003iterating,bauschke2009bregmanC,bauschke2009bregmanK},
the celebrated mirror descent algorithms \cite{beck2003mirror,nemirovskij1983problem},
generalized Bregman distances \cite{burachik2021generalized1} and various properties of the corresponding envelopes and proximal operators \cite{burachik2021generalized2},
and Bregman algorithms for problems beyond traditional Lipschitz smoothness \cite{bauschke2017descent,bolte2018first,lu2018relatively,ahookhosh2021bregman,latafat2022bregman,wang2024mirror,ou2025linesearch},
just to name a few.
As a result, the analysis of Bregman operators has attracted considerable attention in the past decade, beginning with the influential work of \cite{kan2012moreau}.

		\phantomsection
		\addcontentsline{toc}{subsection}{Bregman notions revisited}%
		\subsection*{Bregman notions revisited}
			It is a widespread convention in the literature to consider \(f\) and \(g\) as functions \(\R^n\to\Rinf\), with minimal working assumptions in place of properness and lower semicontinuity \emph{on the whole space \(\R^n\)} \cite{kan2012moreau,laude2020bregman,wang2022bregman}.
On the other hand, it should be noted from \eqref{subeq:prox} and \eqref{subeq:env} that both the proximal mapping and the Moreau envelope solely depend on the values of \(f\) on \(\dom\kernel\) and of \(g\) on \(\interior\dom\kernel\).
As such, it is expectable that conditions solely within these domains should suffice.

In this paper we demonstrate that this is indeed the case, and that a much richer theory can be established by explicitly considering functions \(\func{f}{\X}{\Rinf}\) and \(\func{g}{\Y}{\Rinf}\) defined on suitable subsets \(\X,\Y\subseteq\R^n\): specifically, with
\[
	\X\coloneqq\dom\kernel
\quad\text{and}\quad
	\Y\coloneqq\interior\dom\kernel.
\]
Importantly, not all functions can be extended to the whole space \(\R^n\) while maintaining properties such as lower semicontinuity and properness, see \cref{ex:ln}.
For instance, when \(\dom\kernel\) is not closed, the conventional view may exclude many \emph{relatively weakly convex} functions--those made convex by adding a suitable multiple of \(\kernel\).\footnote{%
	If \(\kernel\) is Legendre and \(\dom\kernel\) is not closed, the relatively weakly convex function \(f\) equal to \(-\kernel\) on \(\dom\kernel\), regardless of how it is extended outside of this set, can never be made globally lsc and proper at the same time.
}
Yet under mild assumptions on \(\kernel\), \emph{all} such excluded functions still enjoy favorable properties: their Bregman-Moreau envelope is smooth, and their proximal mapping is continuous for small parameters.

Our viewpoint of restricted domains does not affect existing results; rather, it \emph{extends} their validity and brings considerable simplifications in the statements, without compromising established notational conventions, see \cref{rem:equiv}.

Ultimately, and leaving the details to the dedicated \cref{sec:Phiconj}, our rationale is also driven by other considerations.
Left and right proximal mappings and envelopes naturally fit into a \emph{\(\Phikernel\)-convexity} picture, as recently explored in \cite{laude2023dualities}.
The theory therein is developed under a blanket full domain assumption of \(\kernel\), and has been then extended in \cite{laude2025anisotropic} beyond this requirement by resorting to \emph{extended arithmetics} to resolve indeterminate expressions such as \(\infty-\infty\).
We argue that extended arithmetics are an unnecessary complication, that can altogether be avoided if one considers functions defined on suitable subsets of \(\R^n\); see our simplified proof of \cite[Prop. 5.11]{laude2023dualities} within \cref{thm:Bcoco}.

But even when \(\X\) is open and thus coincides with \(\Y\), including the case in which \(\X=\R^n\), we still consider natural to treat \(\X\) and \(\Y\) as \emph{distinct} spaces, only dually related through the coupling \(\Phikernel\), and functions \(\func{f}{\X}{\Rinf}\) and \(\func{g}{\Y}{\Rinf}\) as distinct \emph{left} and \emph{right} objects.
Our emphasis in the notation, e.g. in insisting on writing \(x\in\interior\X\) as opposed to \(x\in\Y\), is also driven by this consideration.

		\phantomsection
		\addcontentsline{toc}{subsection}{Contribution}%
		\subsection*{Contribution}
			In this work we challenge the conventional approach of defining functions on the full ambient space \(\R^n\), arguing instead for a domain-restricted perspective that fully reflects the actual behavior of Bregman-type operators.
Through several examples we showcase how our approach not only extends the validity of existing results but it also brings considerable simplifications in their statements.

To do so, we revisit standard notions of variational analysis for functions and operators defined on subsets on \(\R^n\), and identify interesting connections which we believe enjoy an independent appeal.
Several new characterizations are derived, and a rich list of examples demonstrates the necessity of each assumption we impose.
In particular, for the \emph{right} proximal mapping and Moreau envelope, in contrast to much of the existing literature which studies them by first converting them into \emph{left} objects, we develop a direct analysis and offer a new representation involving an \emph{epi-composite} function.

Ultimately, we provide a new characterization of \emph{relative smoothness} \cite{bauschke2017descent,lu2018relatively} that links \emph{Bregman cocoercivity} \cite{dragomir2021methodes,dragomir2021fast} and \emph{anisotropic strong convexity} \cite{laude2023dualities,laude2025anisotropic} inequalities.
We also framed results investigated in this paper in the context of \emph{\(\Phikernel\)-conjugacy}, which is a fundamental notion in optimal transport \cite{figalli2023invitation} (see also \cite{rubinov2013abstract}), and have recently been used to study dualities of relative smoothness and relative strong convexity \cite{laude2023dualities}.

		\phantomsection
		\addcontentsline{toc}{subsection}{Paper organization}%
		\subsection*{Paper organization}
			The remainder of this paper is organized as follows.
We conclude this introductory section with a list of notational conventions adopted throughout.
In \cref{sec:analysis}, we revisit some basics in nonsmooth and variational analysis for functions and set-valued operators defined on subsets of \(\R^n\).
These serve as the groundwork for \cref{sec:Bprox}, where we introduce the central Bregman objects in accordance with our proposed convention, supported by many examples and, in \cref{sec:Phiconj}, also from the perspective of \(\Phikernel\)-conjugacy.
The developments in this last section lead to a novel characterization of \emph{relative smoothness}, that partially closes the gap with known equivalences in the Euclidean setting.

		\phantomsection
		\addcontentsline{toc}{subsection}{Notation}%
		\subsection*{Notation}
		The set of natural numbers is \(\N\coloneqq\set{0,1,2,\dots}\), while \(\R\) and \(\Rinf\coloneqq\R\cup\set{\pm\infty}\) denote the set of real and extended-real numbers, respectively.
We use \(\innprod{{}\cdot{}}{{}\cdot{}}\) to denote the standard inner product on \(\R^n\), and let \(\norm{x}=\sqrt{\innprod{x}{x}}\) be the induced norm.
We also denote \(\j(x)=\frac{1}{2}\norm{x}^2\).
The interior, closure, and boundary of a set \(S\subseteq\R^n\) are respectively denoted as \(\interior S\), \(\closure S\), and \(\boundary S=\closure S\setminus\interior S\).
The \emph{indicator function} of \(S\) is the extended-real-valued function \(\func{\indicator_S}{\R^n}{\Rinf}\) defined as \(\indicator_S(x)=0\) if \(x\in S\) and \(\infty\) otherwise.

	\section{Preliminaries on maps defined on subsets of \texorpdfstring{\(\R^n\)}{Rn}}\label{sec:analysis}%
		This section revisits fundamental notions from convex and variational analysis, adapted to the setting of functions \(\func{f}{X}{\Rinf}\) and set-valued operators \(\ffunc{T}{X}{Y}\), where \(X,Y\subseteq\R^n\) are assumed to be nonempty and convex.
While convexity of \(X\) and \(Y\) is not strictly necessary for many of the definitions presented here, it is required for later developments.

		\subsection{Relative topology}
			Our aim is to extend classical results---typically formulated for the full vector space \(X=Y=\R^n\)---to the more general case where functions and operators are defined on \emph{subsets} of \(\R^n\).
To do so, we adopt a viewpoint that treats \(X\) and \(Y\) as topological spaces in their own right, equipped with the subspace topology inherited from \(\R^n\).
In this sense, we intentionally ``forget'' that \(X\) and \(Y\) are embedded in \(\R^n\), and instead focus on their intrinsic topological properties.

Accordingly, when we refer to a set as being \emph{closed in \(X\)} (or \emph{relative to \(X\)}), we mean that it is the intersection of \(X\) with a closed set in \(\R^n\); the same terminology applies to open sets and neighborhoods.
Note that closedness and sequential closedness coincide also in the induced topologies.
Namely, a set \(C\) is closed in \(X\) if and only if whenever a sequence \(\seq{x^k}\) entirely contained in \(C\) converges to a point in \(X\), this point also belongs to \(C\).

We also remind that a set \(C\) is \emph{compact} if any open cover admits a finite subcover.
That is, if for any family \(\mathcal A\) of open sets such that \(C\subseteq\bigcup_{A\in\mathcal A}A\), there exists a finite subfamily \(\mathcal B\subseteq\mathcal A\) such that \(C\subseteq\bigcup_{A\in\mathcal B}A\).
Importantly, much like boundedness, compactness does not depend on the ambient space:

\begin{lemma}[Compactness is an ``intrinsic'' property]\label{thm:cpt}%
	Let \(C\subseteq X\subseteq\R^n\), and equip \(X\) with the subspace topology inherited from \(\R^n\).
	Then \(C\) is compact in \(X\) if and only if \(C\) is compact in \(\R^n\).
\end{lemma}
\begin{proof}
	Suppose that \(C\) is compact in \(X\).
	Let \(\mathcal A\) be a collection of open sets in \(\R^n\) such that \(C\subseteq\bigcup_{A\in\mathcal A}A\).
	Then, since \(C\subseteq X\) one also has that \(C\subseteq\bigcup_{A\in\mathcal A}(A\cap X)\).
	Each set \(A\cap X\) is open in \(X\), hence there exist finitely many \(A_1,\dots,A_N\in\mathcal A\) such that
	\(
		C
	\subseteq
		\bigcup_{i=1}^N(A_i\cap X)
	\subseteq
		\bigcup_{i=1}^NA_i
	\).
	Thus, we have extracted a finite cover of open sets in \(\R^n\).

	Conversely, suppose that \(C\) is compact in \(\R^n\), and let \(\mathcal B\) be a collection of open sets in \(X\) such that \(C\subseteq\bigcup_{B\in\mathcal B}B\).
	Each of these sets is of the form \(B=A\cap X\) for some open set \(A\) of \(\R^n\), hence
	\(C\subseteq\bigcup_{B\in\mathcal B}B=\bigcup_{A\in\mathcal A}A\) where \(\mathcal A\coloneqq\set{A \text{ open in }\R^n}[A\cap X\in\mathcal B]\) is a collection of open sets in \(\R^n\).
	By compactness in \(\R^n\), there exists finitely many \(A_1,\dots,A_N\in\mathcal A\) such that
	\(
		C
	\subseteq
		\bigcup_{i=1}^NA_i
	\).
	Since \(C\subseteq X\), one also has that
	\(
		C
	\subseteq
		\bigcup_{i=1}^N(A_i\cap X)
	\).
	Since \(A_i\cap X\in\mathcal B\) for all \(i\), we have identified a finite subcover of \(\mathcal B\).
\end{proof}

In light of this equivalence, we may unambiguously talk about compactness without specifying the ambient space.
Thus, compactness of \(C\) is equivalent to closedness and boundedness in \(\R^n\).

		\subsection{Set-valued operators}\label{sec:setvalued}%
			The notation \(\ffunc{T}{X}{Y}\) indicates that \(T\) is a \emph{set-valued operator} mapping points \(x\in X\) to sets \(T(x)\subseteq Y\).
The \emph{graph} of \(T\) is the subset of \(X\times Y\) defined as \(\graph T\coloneqq\set{(x,y)\in X\times Y}[y\in T(x)]\), while its \emph{(effective) domain} is \(\dom T\coloneqq\set{x\in X}[T(x)\neq\emptyset]\) and its \emph{range} is \(\range T\coloneqq\bigcup_{x\in X}T(x)\).
The \emph{inverse} of \(T\) is always defined as the set-valued mapping \(\ffunc{T^{-1}}{Y}{X}\) given by \(T^{-1}(y)=\set{x\in X}[y\in T(x)]\) for every \(y\in Y\).

\begin{definition}[local boundedness]%
	A set-valued operator \(\ffunc{T}{X}{Y}\) is said to be \emph{locally bounded} at a point \(x\in X\) if there exists a neighborhood \(\Nb_x\) of \(x\) in \(X\) such that \(T(\Nb_x)\) is bounded.
	If this holds for every \(x\in X\) then we simply say that \(T\) is \emph{locally bounded}.
\end{definition}

Here and throughout, for a set \(U\subseteq X\) we denote \(T(U)\coloneqq\bigcup_{x\in U}T(x)\).

\begin{definition}[outer semicontinuity]\label{def:osc}%
	A set-valued operator \(\ffunc{T}{X}{Y}\) is said to be \emph{outer semicontinuous} (osc) at a point \(x\in X\) if
	\begin{equation}\label{eq:osc}
		\set{y\in Y}[\exists\seq{x^k,y^k}\to(x,y) \text{ with }(x^k,y^k)\in\graph T\ \forall k]
	\subseteq
		T(x).
	\end{equation}
	When this holds at every \(x\in X\), which is equivalent to \(\graph T\) being closed in \(X\times Y\), we simply say that \(T\) is osc.
\end{definition}

We remark that an operator \(X \rightrightarrows Y\) may be osc relative to the target set \(Y\), but this property can be lost if \(Y\) is extended to the full space \(\R^n\).
The converse, however, always holds: if the operator with target set \(\R^n\) is osc, then it is also osc whenever the target set is shrunk to any subset \(Y\) (containing the range).

Thus, the choice of the target set \(Y\) is crucial.
In fact, the two notions coincide whenever \(Y\) is \emph{complete}, i.e., when \(Y\) is closed in \(\R^n\).
The following example illustrates how the lack of closedness of \(Y\) is precisely what causes the discrepancy described above.

\begin{example}[osc: role of the target space \(Y\)]\label{ex:osc:Y}%
	Consider \(\ffunc{T_1}{[0,1]}{(0,1]}\) and \(\ffunc{T_2}{[0,1]}{[0,1]}\) defined as
	\[
		T_1(x)=T_2(x)
	=
		\begin{cases}
			\set{x} & x\in(0,1]\\
			\set{\tfrac{1}{2}} & x=0.
		\end{cases}
	\]
	While these operators are essentially the same, the difference in the target spaces results in \(T_1\) being osc, but not \(T_2\).
	This is because \(\graph T_1=\graph T_2\) are closed relative to \([0,1]\times (0,1]\), but not relative to \([0,1]\times[0,1]\).

	More specifically, \((x^k,y^k)\in\graph T_1=\graph T_2\) with \(x^k=y^k=\nicefrac{1}{k}\), \(k\in\N_{\geq1}\), violates outer semicontinuity of \(T_2\) as \(k\to\infty\), since \((x^k,y^k)\) converges to \((0,0)\notin\graph T_2\) in \([0,1]\times[0,1]\).
	On the other hand, the same sequence is of no concern to \(T_1\), since it is \emph{not convergent} in the space \([0,1]\times(0,1]\).
\end{example}

The set on the left-hand side of \eqref{eq:osc} corresponds to a Painlevé--Kuratowski outer limit \cite[Def.\ 4.1]{rockafellar1998variational} in the space \(X\times Y\).
Differently from the more common formulation as in \cite[Def.\ 5.4]{rockafellar1998variational} which assumes \(Y = \R^n\), our definition extends the standard notion by explicitly allowing arbitrary \(Y \subseteq \R^n\); see \cite[Sec. 2.5]{burachik2008setvalued} for further generalizations to topological spaces.

While this generalization may seem pedantic, it is conceptually important.
The choice of target space does not affect the validity of our forthcoming results: under minimal working assumptions, the (left) proximal mapping we study will be shown to be outer semicontinuous whether it is viewed as mapping into \(\dom\kernel\) or into the entire \(\R^n\).
Nonetheless, explicitly accounting for a restricted codomain is essential in our framework, which consistently emphasizes reasoning within the \emph{natural domains and codomains} of the operators involved.
In particular, it reinforces our central claim: all relevant topological properties---even those traditionally stated in terms of ambient space behavior---can be fully understood by examining the operator solely within its intrinsic (co)domain of definition.

For future reference, we list two trivial properties whose proof is self-apprent.

\begin{fact}\label{thm:obvious}%
	For \(\ffunc{T}{X}{Y}\) and \(x\in X\), the following hold:
	\begin{enumerate}
	\item \label{thm:lb=>bv}%
		If \(T\) is locally bounded at \(x\), then \(T(x)\) is bounded.
	\item \label{thm:osc=>cv}%
		{\upshape \cite[Rem. 2.5.2]{burachik2008setvalued}} If \(T\) is osc at \(x\), then \(T(x)\) is closed in \(Y\).
	\end{enumerate}
\end{fact}

We next provide another notion of continuity for set-valued operators that, as we shall see, strengthens outer semicontinuity whenever \(T\) maps each \(x\in X\) to a (possibly empty) set \(T(x)\) which is closed in \(Y\).

\begin{definition}[upper semicontinuity]\label{def:usc}%
	\(\ffunc{T}{X}{Y}\) is said to be \emph{upper semicontinuous} (usc) at a point \(x\in X\) if for any open set \(V\supseteq T(x)\) there exists a neighborhood \(\Nb_x\) of \(x\) in \(X\) such that \(T(\Nb_x)\subseteq V\).\footnote{%
		Note that the statement is unaffected if one replaces the set \(V\) with \(V\cap Y\).
		That is, \(V\) can either be open in \(\R^n\) or \emph{relative to \(Y\)}.
	}%
	When this holds at every \(x\in X\), we simply say that \(T\) is usc.
\end{definition}

The interested reader is referred to the comprehensive book \cite{burachik2008setvalued} for a very general account of these and many other properties of set-valued mappings.
We however point out a slight imprecision in the definitions \cite[Def. 2.5.1 and 2.5.15]{burachik2008setvalued} of outer and upper semicontinuity and local boundedness therein, which only apply to points \emph{in the domain} of the operator.
We argue that such definitions do make sense also at points outside of the domain, and that an operator \(\ffunc{T}{X}{Y}\) is to be considered locally bounded, osc or usc when it is so at any point \(x\in X\), and not merely at any \(x\in\dom T\).
For instance, neglecting points outside of the domain would result in the equivalence of outer semicontinuity of \(T\) and the closedness of \(\graph T\) to fail.\footnote{%
	As a simple illustration, let \(X=Y=\R\) and define \(T(x)=\set{0}\) if \(x>0\) and \(T(x)=\emptyset\) otherwise;
	\(T\) would qualify as an osc operator according to \cite[Def. 2.5.1]{burachik2008setvalued}, yet \(\graph T=(0,\infty)\times\set{0}\) is not closed in \(X\times Y\).
	The equivalence between outer semicontinuity and closedness of the graph is explicitly stated in \cite[Thm. 2.5.4]{burachik2008setvalued}; we believe that the incorrect passage in the proof lies in its last sentence, since the contradiction therein arises only provided that \(x\in\dom T\).
}

\begin{lemma}\label{thm:usclb}%
	Suppose that \(\ffunc{T}{X}{Y}\) is usc.
	Then \(T\) is locally bounded at \(x\in X\) if and only if \(T(x)\) is bounded.
\end{lemma}
\begin{proof}
	That local boundedness implies boundedness is obvious, in fact regardless of upper semicontinuity; see \cref{thm:lb=>bv}.
	Conversely, suppose that \(T(x)\) is bounded, and let \(V\supseteq T(x)\) be a bounded open set.
	By upper semicontinuity, there exists a neighborhood \(\Nb_x\) of \(x\) such that \(T(\Nb_x)\subseteq V\), which yields the sought local boundedness of \(T\) at \(x\).
\end{proof}

The key points in the following \cref{thm:oscusc} are special cases of more general statements in \cite{burachik2008setvalued}, which extend \cite[Thm. 5.19]{rockafellar1998variational} by relating outer and upper semicontinuity for operators \(X\rightrightarrows Y\).
A complete proof will require the following observation:

\begin{lemma}\label{thm:domTclosed}%
	For \(\ffunc{T}{X}{Y}\) and a point \(x\in X\setminus\dom T\), the following are equivalent:
	\begin{enumerateq}
	\item \label{thm:domTclosed:usc}%
		\(T\) is usc at \(x\).
	\item \label{thm:domTclosed:int}%
		\(x\) belongs to the interior of \(X\setminus\dom T\) (relative to \(X\)).
	\end{enumerateq}
	When \(Y\) is closed, the following can be added to the equivalence:
	\begin{enumerateq}[resume]
	\item \label{thm:domTclosed:osclb}%
		\(T\) is osc and locally bounded at \(x\).
	\end{enumerateq}
	In particular, if either \(Y\) is closed and \(T\) is osc and locally bounded, or if \(T\) is usc, then \(\dom T\) is closed (relative to \(X\)).
\end{lemma}
\begin{proof}~
	\begin{itemize}
	\item ``\ref{thm:domTclosed:usc} \(\Leftrightarrow\) \ref{thm:domTclosed:int}''
		Obvious (take \(V=\emptyset\) in \cref{def:usc}).
	\item ``\ref{thm:domTclosed:int} \(\Rightarrow\) \ref{thm:domTclosed:osclb}''
		This too is obvious (and holds independently of whether \(Y\) is closed in \(\R^n\)), since \(T(x')=\emptyset\) for all \(x'\in X\) close to \(x\).
	\end{itemize}
	To conclude, suppose that \(Y\) is closed in \(\R^n\).
	\begin{itemize}
	\item ``\ref{thm:domTclosed:osclb} \(\Rightarrow\) \ref{thm:domTclosed:int}''
		Contrary to the claim, suppose that \(x\) is an accumulation point of \(\dom T\), and consider a sequence \(\seq*{x^k,y^k}\) with \(x^k\to x\) and \(y^k\in T(x^k)\) for all \(k\in\N\).
		By local boundedness at \(x\), we may extract a bounded subsequence \(\seq*{y^k}_{k\in K}\) for some index set \(K\subseteq\N\).
		Up to further extracting if necessary, such sequence converges to a point \(y\in\R^n\); the assumption of closedness of \(Y\) ensures that \(y\in Y\) (that is, \(\seq*{y^n}\) also converges in \(Y\)).
		Outer semicontinuity at \(x\) implies that \(y\in T(x)\), contradicting that \(x\notin\dom T\).
	\qedhere
	\end{itemize}\let\qed\relax
\end{proof}

We are now ready to relate outer and upper semicontinuity at any point \(x\in X\); as hinted in \cref{thm:domTclosed}, one direction will require local boundedness and that \(Y\) be closed in \(\R^n\) (so as to make \(Y\) a \emph{complete} topological space); the converse one will instead require that the operator is closed-valued.

\begin{lemma}[osc vs usc]\label{thm:oscusc}%
	For \(\ffunc{T}{X}{Y}\) and \(x\in X\), the following hold:%
	\begin{enumerate}
	\item \label{thm:oscusc:usccv=>osc}%
		{\upshape \cite[Prop. 2.5.21(i)]{burachik2008setvalued}} If \(T\) is usc at \(x\) and \(T(x)\) is closed in \(Y\), then \(T\) is osc at \(x\).
	\item \label{thm:oscusc:osclb=>usc}%
		{\upshape \cite[Prop. 2.5.24(i)]{burachik2008setvalued}} If \(x\in\dom T\) and \(T\) is osc and locally bounded at \(x\), then \(T\) is usc at \(x\).
		When \(Y\) is closed in \(\R^n\), this is also true for \(x\notin\dom T\).%
	\item \label{thm:oscusc:Yclosed}%
		Suppose that \(Y\) is closed in \(\R^n\).
		Then, \(T\) is osc and locally bounded at \(x\) iff \(T\) is usc at \(x\) and \(T(x)\) is compact.
	\end{enumerate}%
\end{lemma}
\begin{proof}~
	\begin{itemize}
	\item ``\ref{thm:oscusc:usccv=>osc}''
		While \cite[Prop. 2.5.21(i)]{burachik2008setvalued} involves global properties, the same proof applies to the localized version stated here.
		For completeness, we include a streamlined argument adapted to our simplified setting.
		Contrary to the claim, suppose that \eqref{eq:osc} is violated at \(x\).
		Then, there exist \(y\in Y\setminus T(x)\) and a sequence \(\graph T\ni(x^k,y^k)\to(x,y)\).
		Since \(T(x)\) is closed in \(Y\), there exists \(\varepsilon>0\) such that \(\cball(y,\varepsilon)\cap T(x)=\emptyset\).
		The set \(V\coloneqq Y\setminus\cball(y,\varepsilon)\) is open in \(Y\) and contains \(T(x)\).
		As such, by upper semicontinuity there exists a neighborhood \(\Nb_x\) of \(x\) in \(X\) such that \(T(\Nb_x)\subseteq V\).
		Up to discarding early iterates, one has that \(x^k\in\Nb_x\) and \(y^k\in\cball(y,\varepsilon)\) for all \(k\).
		The fact that \(y^k\in T(x^k)\subseteq T(\Nb_x)\subseteq V=Y\setminus\cball(y,\varepsilon)\) contradicts the fact that \(y^k\in\cball(y,\varepsilon)\).

	\item ``\ref{thm:oscusc:osclb=>usc}''
		\cite[Prop. 2.5.24(i)]{burachik2008setvalued} shows the claim when \(x\in\dom T\), while the case \(x\notin\dom T\) is covered by \cref{thm:domTclosed}.

	\item ``\ref{thm:oscusc:Yclosed}''
		If \(T\) is osc and locally bounded at \(x\), then it is usc at \(x\) by assertion \ref{thm:oscusc:osclb=>usc}, with \(T(x)\) closed in \(Y\) and bounded by \cref{thm:obvious}.
		Since \(Y\) is closed, \(T(x)\) is also closed relative to \(\R^n\), and thus compact.
		Conversely, if \(T\) is usc at \(x\) and \(T(x)\) is compact, then this set is also bounded and thus \(T\) is locally bounded at \(x\) by \cref{thm:usclb}.
		Since \(T(x)\) is also closed (in \(\R^n\), hence relative to \(Y\) in particular), outer semicontinuity follows from assertion \ref{thm:oscusc:usccv=>osc}.%
	\qedhere
	\end{itemize}\let\qed\relax
\end{proof}

Closedness of \(Y\) plays an important role for the implication in \cref{thm:oscusc:Yclosed} to hold:

\begin{example}[osc vs usc I: closedness of \(Y\)]\label{ex:oscusc:Y}%
	Consider the operator \(T_1\) as in \cref{ex:osc:Y}.
	Despite being osc and closed-valued (in fact, \emph{compact-valued} and locally bounded), \(T_1\) is however not usc at \(x=0\): the open set \(V=(\frac{1}{4},\frac{3}{4})\) contains \(T_1(0)=\set{\frac{1}{2}}\), yet any neighborhood \(\Nb_x\) of \(x\) in \(X\) contains a sufficiently small right interval \([0,\varepsilon)\) whose image under \(T_1\) contains \((0,\varepsilon)\not\subseteq V\).
\end{example}

We will see in the next section that, under minimal working assumptions, the left proximal mapping is usc, locally bounded, and osc (on \(\interior\dom\kernel\)!).
Recall that outer semicontinuity corresponds to the \emph{closedness of the graph} in the product topology of \(\X \times \Y\).
For the proximal mapping, not only is the graph closed in \(\X \times \Y\), but also in the larger space \(\X \times \R^n\).

At first glance, this might seem to undermine our general stance of reasoning solely within the product space \(\X \times \Y\), using the topologies induced from \(\R^n\).
Indeed, in general topology, a set that is closed in a subspace need not be closed in the ambient space---so our framework might appear limited in capturing such global properties.
The next result will demonstrate that this is not the case:
the graph's closedness in the ambient space \(\X \times \R^n\) can, in fact, be characterized from topological properties within \(\X \times \Y\) alone.
The key lies in the \emph{compactness} of the images of the mapping---a topological property that is preserved when passing from a subspace to a superspace.
Since compact sets are closed in \(\R^n\), this allows us to \emph{recover closedness in the ambient space without stepping outside our induced-topology framework}.

\begin{corollary}\label{thm:usccpt}%
	For any \(\ffunc{T}{X}{Y}\), the following are equivalent:
	\begin{enumerateq}
	\item
		\(T\) is usc and compact-valued.
	\item
		\(T\) is locally bounded and \(\graph T\) is closed in \(X\times\R^n\).
	\item
		\(T\) is locally bounded and osc, and remains so even if the target space \(Y\) is extended to the whole \(\R^n\).
	\end{enumerateq}
\end{corollary}
\begin{proof}
	The equivalence of the last two assertions is obvious.
	That of the first two follows from \cref{thm:oscusc:Yclosed} with \(Y\) being replaced by \(\R^n\).
\end{proof}

We have seen in \cref{ex:oscusc:Y} that closedness of \(Y\) is fundamental for outer semicontinuity to imply upper semicontinuity, even if \(T\) is locally bounded and compact valued.
The following example shows how local boundedness is also essential, as well as how changes to the domain of definition \(X\) may affect the validity of this property:

\begin{example}[osc vs usc II: local boundedness]\label{ex:oscusc:lb}%
	Let \(A\subseteq\R\) be any compact set.
	Consider \(\ffunc{T_1}{(0,1)}{\R}\) and \(\ffunc{T_2}{[0,1)}{\R}\), where \(T_2(0)=A\) and
	\[
		T_1(x)=T_2(x)=\set{\tfrac{1}{x}}
	\quad\text{for }
		x\in(0,1).
	\]
	Note that if \(A=\emptyset\) one has that \(\graph T_1=\graph T_2\).
	Regardless of this choice, both operators are osc, but only \(T_1\) is locally bounded.
	As such, \(T_1\) is also usc as ensured by \cref{thm:oscusc:osclb=>usc}.
	On the contrary, failure of local boundedness results in \(T_2\) not being usc, despite it being osc.
\end{example}

\begin{figure}[h]
	\centering
	\fbox{\parbox{0.86\linewidth}{\centering
		\begin{tikzpicture}[
	node distance = 0.25cm and 1cm,
	lwidth/.style = {line width = 0.75pt},
	below*/.style = {below = #1, yshift = -1cm},
	null/.style = {
		box,
		draw = none,
	},
	box/.style = {
		draw,
		inner sep = 5pt,
		rounded corners,
		minimum width = 1.5cm,
		minimum height = 0.75cm,
		align = center,
	},
	arrow/.style = {
		lwidth,
		> = {Stealth[length=5pt]},
		shorten > = 1pt,
		shorten < = 1pt,
	},
	arrow*/.style = {
		red,
		arrow,
		dashed,
	},
	ref*/.style = {
		midway,
		above,
		scale = 0.75,
	},
	ref/.style = {
		ref*,
		sloped,
	},
]

	\node[box] (gph)   {\(\graph\) closed in \(X\times\R^n\)};
	\node[box] (osc)   [right = of gph] {osc};
	\node[box] (lb)    [left  = of gph] {lb};

	\node[null] (center) [below* = of gph] {};
	\coordinate (lb_gph) at ($(lb.south east)!0.5!(gph.south west)$);
	\node[box] (usckv) at (lb_gph |- center) {usc k-v};

	\node[box] (usccv) [above = of center] {usc c-v};
	\node[box] (uscbv) [below = of center] {usc b-v};

	\node[box] (usclb) at (osc |- uscbv) {usc lb};

	\node[box] (osclb) at ($(osc)!0.5!(usclb)$) {osc lb};

	\draw[
		lwidth,
		decorate,
		decoration = {brace, amplitude = 6pt},
	] (usccv.south east) -- (uscbv.north east) node[midway, right, yshift=0.1pt] (usccbv) {};
	\draw[
		lwidth,
		decorate,
		decoration = {brace, amplitude = 6pt},
	] (gph.south west) -- (lb.south east) node[midway, below] (gphlb) {};

	\draw[arrow, ->, bend right = 20] (usckv.south east) to (uscbv.west);
	\draw[arrow, ->, bend left = 20] (usckv.north east) to (usccv.west);
	\draw[arrow, <->] (uscbv.east) -- node[ref]{\cref{thm:usclb}} (usclb.west);
	\draw[arrow, ->, out = 0, in = 135, bend right = 10] (usccv.east) to node[ref, below] {\cite[Prop. 2.5.21(i)]{burachik2008setvalued}} (osc.south west);
	\draw[arrow, ->] (usccbv.east) to[out=0, in=180] (osclb.west);
	\draw[arrow, ->] (osclb.north) -- (osc.south);

	\draw[arrow, -> ] ([yshift =  3pt]gph.east) -- ([yshift =  3pt]osc.west);
	\draw[arrow*,<- ] ([yshift = -3pt]gph.east) -- ([yshift = -3pt]osc.west);

	\draw[arrow*,->, bend left = 25] ([yshift=-1pt]osclb.west) to node[ref, pos=0.7]{\cref{thm:oscusc:Yclosed}} (usckv.east);

	\draw[arrow ,<->] (usckv.north) to[out=90, in=270] node[ref*, anchor = east] {\cref{thm:usccpt}} (gphlb.south);
	\draw[arrow, ->] (osclb.south) -- node[ref*, anchor = west] {\cite[Prop. 2.5.24(i)]{burachik2008setvalued}} (usclb.north);
\end{tikzpicture}

		\begin{tabular}[t]{@{}>{\bf}r<{:~}@{}l >{\bf}r<{:~}@{}l@{}}
			usc & upper semicontinuous
			&
			c-v & closed-valued
		\\
			osc & outer semicontinuous
			&
			b-v & bounded-valued
		\\
			lb & locally bounded
			&
			k-v & compact-valued
		\end{tabular}
	}}%
	\caption{%
		Interrelations among different properties for a set-valued operator \(\ffunc{T}{X}{Y}\) between nonempty subsets \(X,Y\subseteq\R^n\).
		Dashed red arrows represent implications that are valid provided that \(Y\) is closed in \(\R^n\).
		References are omitted when the corresponding implications are either obvious or directly follow from other implications.%
	}%
	\label{fig:oscusc}%
\end{figure}

		\subsection{Extended-real-valued functions}
			The \emph{(effective) domain} of an extended-real-valued function \(\func{f}{X}{\Rinf}\) is the set \(\dom f\coloneqq\set{x\in X}[f(x)<\infty]\), and \(f\) is said to be \emph{proper} if \(f\not\equiv\infty\) and \(f>-\infty\).
The \emph{epigraph} of \(f\) is \(\epi f\coloneqq\set{(x,\alpha)\in X\times\R}[f(x)\leq\alpha]\), and \(f\) is said to be \emph{lower semicontinuous} (lsc) if \(f(x)=\liminf_{x'\to x}f(x')\) holds for every \(x\in X\).

\begin{remark}
	Since \(f\) is \emph{defined} on \(X\), convergence \(x'\to x\) is meant \emph{relative to} \(X\), so that writing \(\liminf_{x'\to x}f(x')\) is to be read as \(\liminf_{X\ni x'\to x}f(x')\).
	Similar conventions will be employed throughout.
\end{remark}

Equivalently, \(f\) is lsc if \(\epi f\) is closed relative to \(X\times\R\).
The \emph{lower semicontinuous hull} of \(f\) is the lsc function \(\func{\closure f}{X}{\Rinf}\) defined as the pointwise supremum among all lsc functions \(X\to\Rinf\) majorized by \(f\), pointwise given by
\[
	\closure f(x)=\liminf_{x'\to x}f(x').
\]
Clearly, \(f\) is lsc iff \(f=\closure f\).

We say that \(f\) is \emph{level bounded} if its \emph{\(\alpha\)-sublevel set} \(\set{x\in X}[f(x)\leq\alpha]\) is bounded for any \(\alpha\in\R\).
The \emph{Fréchet subdifferential} of \(f\) is \(\ffunc{\rsubdiff f}{X}{\R^n}\) given by
\[
	\rsubdiff f(x)
\coloneqq
	\set{v\in\R^n}[
		\textstyle
		\liminf_{\substack{z\to x\\z\neq x}}
		\frac{
			f(z)-f(x)-\innprod{v}{z-x}
		}{
			\norm{z-x}
		}
	\geq
		0
	],
\]
and we call an element \(v\in\rsubdiff f(x)\) a \emph{regular subgradient} of \(f\) at \(x\).
The \emph{Mordukhovich (limiting) subdifferential} of \(f\) is \(\ffunc{\lsubdiff f}{X}{\R^n}\) given by
\[
	\lsubdiff f(x)
\coloneqq
	\set{v\in\R^n}[
		\exists(x^k,v^k)\in\graph\rsubdiff f,
		~
		k\in\N: (x^k,f(x^k),v^k)\to (x,f(x),v)
	]
\]
for \(x\in\dom f\), and \(\lsubdiff f(x)=\emptyset\) otherwise.

		\subsection{Convex analysis}
			The \emph{Fenchel (convex) subdifferential} of \(f\) is \(\ffunc{\fsubdiff f}{X}{\R^n}\), where
\[
	\fsubdiff f(x)
=
	\set{u\in\R^n}[f(x')\geq f(x)+\innprod{u}{x'-x} ~ \forall x'\in X].
\]
In particular, for any \(\bar\xi\in\R^n\) one has that
\begin{equation}\label{eq:Fermat}
	\bar x\in\argmin\bigl(f - \innprod{\bar\xi}{{}\cdot{}}\bigr)
\quad\Leftrightarrow\quad
	\bar\xi\in\fsubdiff f(\bar x)
\quad\Leftrightarrow\quad
	\bar x\in(\fsubdiff f)^{-1}(\bar\xi).
\end{equation}
The \emph{convex hull} of \(f\), denoted \(\conv f\), is the pointwise supremum among all convex functions \(X\to\Rinf\) majorized by \(f\), while the supremum of all such functions which are additionally lsc is the \emph{closed convex hull}, denoted \(\conv*f\).
In particular, one always has that
\begin{equation}\label{eq:convleq}
	\conv*f\leq\conv f\leq f
\quad\text{and}\quad
	\conv*f\leq\closure f\leq f.
\end{equation}

The \emph{convex hull} of a set \(E\subseteq\R^n\), also denoted \(\conv E\), is the intersection of all convex subsets of \(\R^n\) containing \(E\).\footnote{%
	Unlike the notions of closedness or openness---which depend on the underlying topology and therefore require a notion relative to \(X\) or \(Y\)---no such distinction is necessary for convexity.
	Since we assume that \(X\) and \(Y\) are convex sets themselves, any subset of \(X\) (or \(Y\)) that is convex in the usual sense is automatically ``convex relative to \(X\)'' (or \(Y\)), and vice versa.
}
The notation \(\relint E\) indicates the \emph{relative interior} of \(E\), namely the interior of \(E\) relative to its affine hull (in \(\R^n\)).

We also remind the definition of the \emph{Fenchel conjugate} of a function \(\func{h}{\R^n}{\Rinf}\) (defined on the whole space \(\R^n\)) as the convex and lsc function \(\func{h^*}{\R^n}{\Rinf}\) given by \(h^*(y)\coloneqq\sup_{x\in\R^n}\set{\innprod{x}{y}-h(x)}\), and that of its \emph{Fenchel biconjugate} \(h^{**}\coloneqq(h^*)^*\).
We also remind that \cite[Prop. 13.45]{bauschke2017convex}
\begin{equation}\label{eq:h**}
	\conv*h=h^{**}.
\end{equation}
We say that \(h\) is \emph{coercive} if \(\liminf_{\norm{x}\to\infty}h(x)=\infty\) (equivalently, if \(h\) is level bounded), and \emph{1-coercive} if the stronger property \(\liminf_{\norm{x}\to\infty}\frac{h(x)}{\norm{x}}=\infty\) holds.

Finally, throughout this work, we adopt a slight abuse of notation by defining the sum of a function \(\func{h}{\R^n}{\Rinf}\) and a function \(\func{f}{X}{\Rinf}\) as the function \(\func{(f+h)}{X}{\Rinf}\) given by \(f+h\coloneqq f+h\restr_X\), with similar conventions extending to analogous operations.

		\subsection{Canonical extensions to the whole space \texorpdfstring{\(\R^n\)}{Rn}}\label{sec:ext}%
			There is a natural way to extend set-valued operators \(X\rightrightarrows Y\) and extended-real-valued functions \(X\to\Rinf\) to the entire space \(\R^n\) while maintaining their fundamental properties.

\begin{definition}[canonical extension of a set-valued operator]\label{def:Text}%
	The \emph{canonical extension} of \(\ffunc{T}{X}{Y}\) is \(\ffunc{\tilde T}{\R^n}{\R^n}\) given by
	\[
		\tilde T(x)
	=
		\begin{cases}
			T(x) & \text{if } x\in X,\\
			\emptyset & \text{otherwise.}
		\end{cases}
	\]
\end{definition}

Note that \(T\) and \(\tilde T\) have same graph, (effective) domain, and range, and moreover one is locally bounded iff the other is.
In fact, it can be easily verified that \(\tilde T\) is the only extension \(\R^n\rightrightarrows\R^n\) preserving the graph, and consequently all these properties.
An important difference between \(T\) and \(\tilde T\), however, is that \(\tilde T\) may fail to be osc even when \(T\) is (consider \(\ffunc{T}{(0,1)}{\R}\) where \(T(x)=\set{0}\) for all \(x\in(0,1)\)).
Clearly, on the other hand, when \(\tilde T\) is osc then so is \(T\).
The restriction of the target set \(Y\), instead, affects the definition of the inverse, having \(\ffunc{T^{-1}}{Y}{X}\) as opposed to \(\ffunc{\tilde T^{-1}}{\R^n}{\R^n}\)
(incidentally, note that the inverse of the canonical extension is the canonical extension of the inverse).

A similar extension can be defined for extended-real-valued functions.

\begin{definition}[canonical extension of an extended-real-valued function]\label{def:fext}%
	The \emph{canonical extension} of \(\func{f}{X}{\Rinf}\) is \(\func{\tilde f}{\R^n}{\Rinf}\) given by
	\[
		\tilde f(x)
	=
		\begin{cases}
			f(x) & \text{if } x\in X,\\
			\infty & \text{otherwise.}
		\end{cases}
	\]
\end{definition}

Functions \(f\) and \(\tilde f\) have the same epigraph, effective domain, sublevel sets, and they share properties such as properness, convexity, and/or level boundedness.
Similarly to the operator case, \(\tilde f\) is easily seen to be the only extension \(\R^n\to\Rinf\) that preserves the epigraph, and in turn all the other properties.
Moreover, note that \(\fsubdiff\tilde f\), \(\rsubdiff\tilde f\), and \(\lsubdiff\tilde f\) are the respective canonical extensions of \(\fsubdiff f\), \(\rsubdiff f\), and \(\lsubdiff f\), in the operator sense of \cref{def:Text}.
Functions \(f\) and \(\tilde f\) also share the convex hull, but \emph{not} the \emph{closed} convex hull.
A primary difference among the two is indeed to be found in the notion of lower semicontinuity (consider \(\func{f}{(0,1)}{\Rinf}\) where \(f(x)=0\) for all \(x\in(0,1)\)), and all the derived properties, such as the lower semicontinuous hull and the closed convex hull.
That is, there can be lsc functions \(\func{f}{X}{\Rinf}\) whose canonical extension is not lsc (but not the other way round).

The following lemma nevertheless clarifies that the hulls of \(f\) and its canonical extension \(\tilde f\) coincide on the set \(X\).
This simple yet important observation allows us to apply known results, originally formulated for functions defined on the entire space \(\R^n\), within our specific setting.

\begin{lemma}\label{thm:fext}%
	Let \(\func{f}{X}{\Rinf}\) be proper and let \(\func{\tilde f}{\R^n}{\Rinf}\) be its canonical extension as in \cref{def:fext}.
	Then, the following hold:
	\begin{enumerate}
	\item \label{thm:clext}%
		\(\closure f=(\closure\tilde f)\restr_X\).
		In particular, \(f\) is lsc iff there exists an lsc (not necessarily proper) function \(\func{\hat f}{\R^n}{\Rinf}\) such that \(f=\hat f\restr_X\).
	\item \label{thm:cvxext}%
		\(
			\tilde f^{**}\restr_X=(\conv*\tilde f)\restr_X=\conv*f
		\leq
			\conv f=(\conv\tilde f)\restr_X
		\).
		The inequality holds as equality whenever \(\tilde f\) is 1-coercive.
	\item \label{thm:clcvx}%
		\(f\) is convex iff \(\tilde f\) is convex.
		In this case, one has that \(\conv*f=\closure f\).
	\end{enumerate}
\end{lemma}
\begin{proof}~
	\begin{itemize}
	\item ``\ref{thm:clext}''
		For \(x\in X\) one has that
		\[
			\closure\tilde f(x)
		\defeq
			\liminf_{x'\to x}\tilde f(x')
		=
			\liminf_{X\ni x'\to x}\tilde f(x')
		=
			\liminf_{x'\to x}f(x')
		\defeq
			\closure f(x),
		\]
		where the second identity follows from the fact that \(\tilde f\equiv\infty\) on \(\R^n\setminus X\), and the third one from the fact that \(\tilde f\restr_X=f\).
		The claim follows from the arbitrariness of \(x\in X\).

	\item ``\ref{thm:cvxext}''
		The inequality was observed in \eqref{eq:convleq}; the first identity follows from \eqref{eq:h**} and the last one is obvious.
		We now show that \(\conv*f=(\conv*\tilde f)\restr_X\).
		Observe that, as a function defined on \(X\), \((\conv*\tilde f)\restr_X\) is convex and lsc, and \((\conv*\tilde f)\restr_X\leq\tilde f\restr_X=f\) by definition.
		As such, \((\conv*\tilde f)\restr_X\leq\conv*f\).
		To show the converse inequality, let \(\func{g}{X}{\Rinf}\) be convex, lsc and such that \(g\leq f\), and let \(\func{\tilde g}{\R^n}{\Rinf}\) be its canonical extension.
		It follows from assertion \ref{thm:clext} that \(g=(\closure\tilde g)\restr_X\).
		Moreover, since \(\tilde g\) is convex (on \(\R^n\)), \(\closure\tilde g\) is convex and lsc by \cite[Cor. 9.10]{bauschke2017convex}, and globally supports \(\tilde f\).
		Consequently \(\conv*\tilde f\geq\closure\tilde g\), resulting in
		\[
			(\conv*\tilde f)\restr_X
		\geq
			(\closure\tilde g)\restr_X
		=
			g.
		\]
		Taking the supremum among all such functions \(g\) yields the sought inequality \(\conv*\tilde f\restr_X\geq\conv*f\).
		When \(\tilde f\) is 1-coercive, by invoking \cite[Lem. 3.3(i)]{benoist1996what} we conclude that \(\conv*\tilde f=\conv\tilde f\) in this case, and in particular equality holds on \(X\).

	\item ``\ref{thm:clcvx}''
		The equivalence of convexity of \(f\) and of \(\tilde f\) is straightforward.
		Next, suppose that \(f\) is convex.
		The claimed identity is true when \(X=\R^n\), see \cite[Cor. 9.10]{bauschke2017convex}, and in particular \(\conv*\tilde f=\closure\tilde f\).
		Combined with the previous assertions the claim follows.
	\qedhere
	\end{itemize}\let\qed\relax
\end{proof}

\begin{fact}[{\cite[Prop. 1.4.3]{hiriarturruty1996convex}}]\label{fact:fsubdiff}%
	Let \(\func{f}{X}{\Rinf}\) be proper and \(\bar x\in X\) be fixed.
	Consider the following statements:
	\begin{enumerateq}
	\item \label{fact:fsubdiff:1}%
		\(\fsubdiff f(\bar x)\neq\emptyset\);
	\item \label{fact:fsubdiff:2}%
		\(f(\bar x)=\conv* f(\bar x)\in\R\);
	\item \label{fact:fsubdiff:3}%
		\(\fsubdiff f(\bar x)=\lsubdiff(\conv* f)(\bar x)\).
	\end{enumerateq}
	One has \ref{fact:fsubdiff:1} \(\Rightarrow\) \ref{fact:fsubdiff:2} \(\Rightarrow\) \ref{fact:fsubdiff:3}.
\end{fact}

As anticipated, the validity of \cref{fact:fsubdiff} for general functions \(X\to\Rinf\), as opposed to \(\R^n\to\Rinf\) only as considered in \cite{hiriarturruty1996convex}, is guaranteed by \cref{thm:fext}, since for any \(\bar x\in X\) one can replace any occurrence of \(f\) therein with the canonical extension \(\tilde f\).
Next, we provide a relation among the Fenchel subdifferentials of \(f\), \(\closure f\), and \(\conv*f\):
\begin{equation}\label{eq:fsubdiffs}
	\fsubdiff f(x)\subseteq\fsubdiff[\closure f](x)\subseteq\fsubdiff[\conv*f](x).
\end{equation}
Indeed, to see the first inclusion in \eqref{eq:fsubdiffs}, notice that for \(v\in\fsubdiff f(x)\) the function \(\X\ni x'\mapsto f(x)+\innprod{v}{x'-x}\) is an lsc minorant of \(f\).
By definition of \(\closure f\), it must also be a minorant of \(\closure f\), and as such one has that
\[
	\closure f(x')\geq f(x)+\innprod{v}{x'-x}
\quad
	\forall x'\in\X.
\]
The inequality holds in particular for \(x'=x\), which yields that \(f(x)\geq\closure f(x)\geq f(x)\).
It follows that \(f(x)\) in the above inequality can be replaced by \(\closure f(x)\), proving the sought inclusion \(v\in\fsubdiff[\closure f](x)\).
The second inclusion in \eqref{eq:fsubdiffs} can be justified in a similar manner.

\begin{lemma}\label{thm:argmincl}%
	For any function \(\func{f}{X}{\Rinf}\) with \(X\subseteq\R^n\) nonempty, one has that \(\inf f=\inf\closure f=\inf\conv*f\) and \(\argmin f\subseteq\argmin\closure f\subseteq\argmin\conv*f\).
\end{lemma}
\begin{proof}
	That \(\inf\conv*f\leq\inf\closure f\leq\inf f\) is obvious from \eqref{eq:convleq}, and equality thus must hold if the latter is \(-\infty\).
	If instead \(\inf f\) is finite, then the constant function \(\inf f\) is convex and lsc and is globally majorized by \(f\); as such, \(\inf f\leq\conv f\) and consequently \(\inf f\leq\inf\conv* f\), completing the proof.

	Finally, the claimed inclusions are apparent from \eqref{eq:fsubdiffs} and \eqref{eq:Fermat}.
\end{proof}

	\section{Bregman proximal map and Moreau envelope}\label{sec:Bprox}
		Throughout the paper, we will consider a function \(\kernel\) complying with the following minimal working assumption, assumed throughout without further mention; additional requirements will be invoked when needed.

\begin{center}
	\parbox{0.8\linewidth}{%
		\begin{tcolorbox}[
			colback = MidnightBlue!3,
			colframe = MidnightBlue!50,
			top = .5ex,
			bottom = .5ex,
			left = .25em,
			right = .25em,
		]%
			Function \(\func{\kernel}{\R^n}{\Rinf}\) is a \emph{distance generating function (dgf)}, namely, proper, lsc, convex, and differentiable on \(\interior\dom\kernel\neq\emptyset\).
			We denote \(\X\coloneqq\dom\kernel\) and \(\Y\coloneqq\interior\dom\kernel\), and similarly \(\X*\coloneqq\dom\kernel*\) and \(\Y*\coloneqq\interior\dom\kernel*\).
		\end{tcolorbox}%
	}%
\end{center}

The symbols \(\X\) and \(\Y\) are introduced for notational conciseness, and will also be helpful later when highlighting a parallel with the general theory of \(\Phikernel\)-convexity; see \cref{sec:Phiconj}.
The name dgf owes to the fact that \(\kernel\) induces a pseudo-distance on \(\R^n\) given by
\begin{equation}\label{eq:D}
	\D(x,y)
=
	\begin{cases}
		\kernel(x)-\kernel(y)-\innprod{\nabla\kernel(y)}{x-y} & \text{if } y\in\interior\dom\kernel
	\\
		\infty & \text{otherwise.}
	\end{cases}
\end{equation}
Apparently, \(\D\) quantifies the gap between the value of \(\kernel\) in its first argument and the first-order approximation at the second, which is zero when \(x=y\in\interior\dom\kernel\) and is otherwise always greater than or equal to zero due to the convexity assumed on \(\kernel\).
The same conclusion can be deduced from the identity \(\kernel(y)+\kernel*(\nabla\kernel(y))=\innprod{y}{\nabla\kernel(y)}\) holding for \(y\in\Y\), thanks to which we can interpret
\begin{equation}\label{eq:DFY}
	\D(x,y)
=
	\kernel(x)+\kernel*(\nabla\kernel(y))-\innprod{\nabla\kernel(y)}{x}
\qquad
	y\in\Y
\end{equation}
as a Fenchel duality gap.
This metric is widely known as the \emph{Bregman distance} induced by \(\kernel\), though it generally lacks symmetry and does not satisfy the triangular inequality, a reason why the term \emph{Bregman divergence} is often preferred.

Additional assumptions on \(\kernel\) induce favorable properties on \(\D\) and on related objects, and will sometimes be needed in our results.
Among these we mention \emph{Legendreness}.

\begin{definition}[Legendre function]%
	We say that \(\kernel\) is \emph{of Legendre type} (or simply \emph{Legendre}) if it is
	\begin{enumerate}
	\item
		\emph{essentially smooth}, namely differentiable on \(\interior\dom\kernel\neq\emptyset\) and such that \(\norm{\nabla\kernel(x^k)}\to\infty\) whenever \(\interior\dom\kernel\ni x^k\to x\in\boundary\dom\kernel\), and
	\item
		\emph{essentially strictly convex}, namely strictly convex on every convex subset of \(\dom\lsubdiff\kernel\).
	\end{enumerate}
\end{definition}

We next recall two fundamental characterizations of essential smoothness and strict convexity.

\begin{fact}[{\cite[Thm. 26.1 and 26.3]{rockafellar1970convex}}]%
	The following are equivalent:
	\begin{enumerateq}
	\item
		\(\kernel\) is essentially smooth;
	\item
		\(\lsubdiff\kernel\) is at most single valued;
	\item
		\(\kernel*\) is essentially strictly convex.
	\end{enumerateq}
	Under any of the above equivalent conditions, one also has that \(\dom\lsubdiff\kernel=\interior\dom\kernel\) and \(\lsubdiff\kernel(x)=\set{\nabla\kernel(x)}\) for any \(x\in\interior\dom\kernel\).
\end{fact}

\begin{fact}[{\cite[Thm. 26.5]{rockafellar1970convex}} and {\cite[Thm. 3.7(v)]{bauschke1997legendre}}]\label{thm:D*}%
	Function \(\kernel\) is of Legendre type iff its conjugate \(\kernel*\) is.
	In this situation, \(\func{\nabla\kernel}{\interior\dom\kernel}{\interior\dom\kernel*}\) is a bijection with inverse \(\nabla\kernel*\), and the following identity holds:
	\[
		\D(x,y)
	=
		\D*(\nabla\kernel(y),\nabla\kernel(x))
	\quad
		\forall x,y\in\interior\dom\kernel.
	\]
\end{fact}

		\subsection{Definitions}
			We begin by introducing well-known objects that comprise the basic toolbox of the paper.
It will be implied without mention that all the subsequent definitions are relative to the dgf \(\kernel\) and, when applicable, to a stepsize parameter \(\lambda>0\).

\begin{definition}[Bregman proximal map]\label{def:prox}%
	The \emph{left (Bregman) proximal mapping} of \(\func{f}{\X}{\Rinf}\) is \(\ffunc{\prox_{\lambda f}}{\Y}{\X}\) defined by
	\begin{subequations}
		\begin{align}
			\prox_{\lambda f}(\bar y)
		\coloneqq{} &
			\argmin_{x\in\X}\set{f(x)+\tfrac{1}{\lambda}\D(x,\bar y)}.
		\intertext{%
			The \emph{right (Bregman) proximal mapping} of \(\func{g}{\Y}{\Rinf}\) is \(\ffunc{\prox*_{\lambda g}}{\X}{\Y}\), where%
		}
			\prox*_{\lambda g}(\bar x)
		\coloneqq{} &
			\argmin_{y\in\Y}\set{g(y)+\tfrac{1}{\lambda}\D(\bar x,y)}.
		\end{align}
	\end{subequations}
\end{definition}

Following the discussion opened in the introduction, we remark that the definitions of Bregman proximal mapping also apply to functions \(\func{f,g}{\R^n}{\Rinf}\) defined on all the space \(\R^n\).
In fact, in this case the definitions are unaffected by removing the constraints over \(\X\) and \(\Y\) in the respective minimization problems, owing to the convention that \(\D(x,y)=\infty\) whenever \(x\notin\X\) or \(y\notin\Y\).
Our explicit emphasis on \(\X\) and \(\Y\) is chosen to highlight that \(\prox_{\lambda f}\) and \(\prox*_{\lambda g}\) solely depend on the value of \(f\) on \(\X\) and of \(g\) on \(\Y\), and once again will find an independent appeal when framing the discussion under the more general lens of \(\Phikernel\)-convexity.

One concrete advantage of our convention is that, as we will show, key properties of \(f\) and \(g\)---such as lower semicontinuity or properness---need only hold relatively to the sets \(\X\) and \(\Y\), rather than on the entire space \(\R^n\) as is typically assumed in the literature.
As \cref{thm:ext,ex:ln} demonstrate, this local perspective allows us to encompass a broader class of functions.

\begin{remark}\label{rem:equiv}%
	In order to keep consistency with the literature, one may regard \(f\) rather as the \emph{equivalence class} \([f]_\sim\) of functions \(X\to\Rinf\) with \(\X\subseteq X\subseteq\R^n\) that agree with \(f\) on \(\X\).
	This choice entitles us to unambiguously adopt the same notation, say, \(\prox_{\lambda f}\), both for \(\func{f}{\R^n}{\Rinf}\) and \(\func{f}{\X}{\Rinf}\).
	The same comments apply to \emph{right} functions \(\func{g}{\Y}{\Rinf}\) as well.
	Regardless of the chosen perspective, in all our results out of clarity we will always explicitly declare the domains of definitions of functions and operators.
\end{remark}

\begin{definition}[Bregman-Moreau envelopes]\label{def:env}%
	The \emph{left (Bregman-Moreau) envelope} of \(\func{f}{\X}{\Rinf}\) is \(\func{\env{f}}{\Y}{\Rinf}\) defined by
	\begin{subequations}
		\begin{align}
			\env{f}(\bar y)
		\coloneqq{} &
			\inf_{x\in\X}\set{f(x)+\tfrac{1}{\lambda}\D(x,\bar y)}.
		\intertext{%
			The \emph{right (Bregman--Moreau) envelope} of \(\func{g}{\Y}{\Rinf}\) is \(\func{\env*{g}}{\X}{\Rinf}\), where%
		}
			\env*{g}(\bar x)
		\coloneqq{} &
			\inf_{y\in\Y}\set{g(y)+\tfrac{1}{\lambda}\D(\bar x,y)}.
		\end{align}
	\end{subequations}
\end{definition}

Consistent with existing literature, the definitions of \(\env{f}\) and \(\env*{g}\) can be extended to the entire space \(\R^n\) by setting them to \(\infty\) outside of \(\Y\) and \(\X\), respectively.
This can indeed be done without loss of generality and notational consistency, up to resorting to the ``equivalence class'' interpretation mentioned in \cref{rem:equiv}.
We nevertheless argue that these envelopes are intrinsically mappings defined on \(\Y\) and \(\X\):
regardless of any extension, their properties are meaningful only relative to these sets.
For example, under assumptions on \(\kernel\) and on \(f\) (on \(\X\)), when viewed as a function \(\Y\rightarrow\Rinf\) and as an operator \(\Y\rightrightarrows\X\), \(\env{f}\) and \(\prox_{\lambda f}\) can be shown to be lower semicontinuous and outer semicontinuous, respectively; see \cref{thm:mainprop}.
If we instead consider them as defined on the entire space \(\R^n\), any such claims must be qualified by stating that the properties hold only relative to \(\Y\).

In this respect, we point out the imprecise statements of \cite[Thm. 2.2(iii) and 2.6]{kan2012moreau} or \cite[Fact 2.2]{wang2022bregman} in which such restrictions are not explicitly mentioned, or the typo in \cite[Def. 2.3]{kan2012moreau} and \cite[Def. 2.1]{wang2022bregman} from which \cref{def:PB} in the next subsection is taken, for it allows, in our notation, any \(y\in\R^n\) as opposed to \(y\in\Y\);
see also the commentary after \cref{thm:mainprop}.

\begin{definition}[Bregman proximal hull]
	The \emph{left Bregman proximal hull} of an extended-real-valued function \(\func{f}{\X}{\Rinf}\) is the function \(\func{\hull_{\lambda}{f}}{\X}{\Rinf}\) given by \(\hull_{\lambda}{f}\coloneqq-\env*{(-\env{f})}\).
	Similarly, the \emph{right Bregman proximal hull} of \(\func{g}{\Y}{\Rinf}\) is \(\func{\hull*_{\lambda}{g}}{\Y}{\Rinf}\) given by \(\hull*_{\lambda}{g}\coloneqq-\env{(-\env*{g})}\).
\end{definition}

When \(\kernel=\j\), for any \(\func{f}{\R^n}{\Rinf}\) these objects reduce to the \emph{proximal hull} \(\Ehull_{\lambda}{f}=-\Eenv_{\lambda}{(-\Eenv_{\lambda}{f})}\), which is the largest function \(h\leq f\) such that \(h+\lambda^{-1}\j\) is convex \cite[Ex. 1.44]{rockafellar1998variational}.

Left and right Bregman Moreau envelopes can be elegantly interpreted as (negative) \emph{left and right conjugates} in the broad theory of \emph{\(\Phikernel\)-convexity}.
We will quickly outline this parallel in the dedicated \cref{sec:Phiconj}, and refer the reader to \cite{laude2023dualities} for a detailed account and the many interesting implications that this perspective offers.
By interpreting the Bregman proximal hulls with the corresponding \emph{biconjugate} operators---as we shall in \cref{sec:Phiconj}, see \eqref{eq:Phi:trienv}---the relations
\begin{equation}\label{eq:triconj}
	\env{f}=\env{(\hull_{\lambda}{f})}
\quad\text{and}\quad
	\env*{g}=\env*{(\hull*_{\lambda}{g})},
\end{equation}
holding for any dgf \(\kernel\), \(\func{f}{\X}{\Rinf}\), and \(\func{g}{\Y}{\Rinf}\), bear the same familiarity as the standard identity \(h^*=(h^{**})^*\) in convex analysis for any function \(\func{h}{\R^n}{\Rinf}\).
A detailed proof of \eqref{eq:triconj} can be found in \cite[Prop. 2.14]{wang2022bregman}, where these generalized conjugacy relations naturally emerge, though not explicitly labeled as such.
While the reference assumes 1-coercivity and Legendreness of \(\kernel\) throughout, the cited proof works with no restriction on either \(\kernel\) or \(f\), as we will confirm in \cref{sec:Phiconj}.

		\subsection{(Left) prox-boundedness and consequences}
			We mention a well-known property which, coupled with 1-coercivity of \(\kernel\), constitutes a fundamental prerequisite for the left proximal mapping to be a meaningful object in practice.

\begin{definition}[\(\kernel\)-prox-boundedness {\cite[Def. 2.3]{kan2012moreau}}]\label{def:PB}%
	We say that \(\func{f}{\X}{\Rinf}\) is \emph{\(\kernel\)-prox-bounded} if there exist \(\lambda>0\) and \(y\in\Y\) such that \(\env{f}(y)>-\infty\).
	The supremum \(\pb\) of all such \(\lambda\) is the \emph{\(\kernel\)-prox-boundedness threshold} of \(f\).
\end{definition}

The above property is useful in combination with 1-coercivity of \(\kernel\), as it ensures, among other things, the nonemptiness of \(\prox_{\lambda f}(\bar x)\) for any sufficiently small \(\lambda\) and \(\bar x \in \X\).
Since these implications are specific to the \emph{left} proximal operator and the \emph{left} Bregman-Moreau envelope, it would be more precise to speak in terms of \emph{left} \(\kernel\)-prox-boundedness.
However, we retain the original terminology, as it is the one most commonly used in the literature.
Expanding upon and generalizing \cite{laude2020bregman,laude2021lower}, we will elaborate in \cref{sec:pb*} on the \emph{right} counterpart, after some necessary intermediate results are proven.

The main consequences of \(\kernel\)-prox-boundedness are listed in \cref{thm:mainprop}, and the proof referred to results in \cite{kan2012moreau}.
Before that, the following preliminary result reassures us that the referenced proofs can correctly be invoked in our more general setting, even for those functions \(\func{f}{\X}{\Rinf}\) which are proper and lsc on \(\X\) but may not possess an extension preserving both properties on the entire space \(\R^n\), such as the one in \cref{ex:ln}.
As we shall see, while the canonical extension \(\tilde f\) of \(f\) as in \cref{def:fext} may indeed fail to be lsc, \emph{the sum \(\lambda\tilde f+\kernel\) is always lsc as long as \(\lambda<\pb\)}.
This generalizes \cite[Prop. 2.4]{wang2022bregman}, which showed that this sum is 1-coercive whenever \(\kernel\) itself is 1-coercive.
Our revisited proof confirms that this remains true even without standing assumptions of Legendreness.

\begin{lemma}\label{thm:ext}%
	Fix \(\func{f}{\X}{\Rinf}\) and extend it to \(\func{\tilde f}{\R^n}{\Rinf}\) by setting it to \(\infty\) on \(\R^n\setminus\X\) as in \cref{def:fext}.
	Then, for any \(y\in\Y\) and \(\lambda>0\) one has
	\[
		\prox_{\lambda f}(y)=\prox_{\lambda\tilde f}(y)
	\quad\text{and}\quad
		\env{f}(y)=\env{\tilde f}(y).
	\]
	When \(f\) is \(\kernel\)-prox-bounded, then the following also hold for any \(\lambda\in(0,\pb)\):
	\begin{enumerate}
	\item \label{thm:ext:lsc}%
		\textup{\cite[Prop. 2.4(i)]{wang2022bregman}}
		If \(\kernel\) is 1-coercive, then \(\lambda\tilde f+\kernel\) too is 1-coercive.
	\item \label{thm:ext:1co}%
		If \(f\) is lsc (on \(\X\)), then \(\lambda\tilde f+\kernel\) too is lsc (on \(\R^n\)).
	\end{enumerate}
\end{lemma}
\begin{proof}
	The identities hold trivially since \(\D(x,y)=\infty\) whenever \(x\notin\X\).
	We next assume that \(f\) is \(\kernel\)-prox-bounded and prove the numbered claims.
	\begin{itemize}
	\item ``\ref{thm:ext:lsc}''
		Let \(\lambda_0\in(\lambda,\pb)\) be fixed, and let \(y_0\in\Y\) be such that \(\env_{\lambda_0}{f}(y^0)>-\infty\).
		Note that the existence of such \(y^0\) ensures that \(f>-\infty\) must hold; moreover, if \(f\equiv\infty\) then all the claims are trivial.
		In what follows, we shall thus suppose that \(f\) is proper.
		To streamline the notation, let us denote \(\psi_{\lambda}\coloneqq\tilde f+\tfrac{1}{\lambda}\D({}\cdot{},y^0)\).
		Since \(\psi_{\lambda}\) and \(\tilde f+\lambda^{-1}\kernel\) differ by an affine term, it suffices to show the claims for \(\psi_{\lambda}\).
		For any \(x\in\R^n\) it holds that
		\begin{align*}
			\psi_{\lambda}(x)
		={} &
			\tilde f(x)
			+
			\tfrac{1}{\lambda_0}\D(x,y^0)
			+
			\left(\tfrac{1}{\lambda}-\tfrac{1}{\lambda_0}\right)\D(x,y^0)
		\\
		\geq{} &
			\env_{\lambda_0}{f}(y^0)
			+
			\left(\tfrac{1}{\lambda}-\tfrac{1}{\lambda_0}\right)\D(x,y^0)
		\\
		={} &
			\left(\tfrac{1}{\lambda}-\tfrac{1}{\lambda_0}\right)\kernel(x)
			+
			\innprod{v}{x}+c,
		\numberthis\label{eq:psigeq}
		\end{align*}
		where
		\(
			c
		\coloneqq
			\env_{\lambda_0}{f}(y^0)
			+
			\bigl(\tfrac{1}{\lambda}-\tfrac{1}{\lambda_0}\bigr)
			\kernel*(\nabla\kernel(y^0))
		\)
		and
		\(
			v
		\coloneqq
			\bigl(\tfrac{1}{\lambda}-\tfrac{1}{\lambda_0}\bigr)
			\nabla\kernel(y^0)
		\).
		From this expression it is clear that whenever \(\kernel\) is 1-coercive, then so is \(\psi_{\lambda}\).

	\item ``\ref{thm:ext:1co}''
		Let \(\alpha\in\R\) and \(\seq{x^k}\) be a sequence converging to a point \(\bar x\) and such that \(\psi_{\lambda}(x^k)\leq\alpha\) for all \(k\).
		To arrive to a contradiction, suppose that \(\psi_{\lambda}(\bar x)>\alpha\).
		Since \(\psi_{\lambda}\) is lsc on \(\X\) by assumption, one must have that \(\bar x\notin\X\), that is, \(\kernel(\bar x)=\infty\).
		On the other hand, for every \(k\) it follows from \eqref{eq:psigeq} that
		\(
			(\tfrac{1}{\lambda}-\tfrac{1}{\lambda_0})\kernel(x^k)
			+
			\innprod{v}{x^k}
			+
			c
		\leq
			\alpha
		\).
		The fact that
		\(
			(\tfrac{1}{\lambda}-\tfrac{1}{\lambda_0})\kernel(\bar x)
			+
			\innprod{v}{\bar x}
			+
			c
		=
			\infty
		>
			\alpha
		\)
		contradicts the fact that \(\kernel\) is lsc (on \(\R^n\)).
	\qedhere
	\end{itemize}\let\qed\relax
\end{proof}

Having ruled out possible lower semicontinuity issues, we are now ready to list the anticipated consequences of \(\kernel\)-prox-boundedness and refer the proofs to existing results.
In the rest of the subsection we will provide some comments and examples to demonstrate the importance of the assumptions in place.

\begin{fact}[{\cite[Thm. 2.2 and 2.6, Cor. 2.2]{kan2012moreau}}]\label{thm:mainprop}%
	Suppose that \(\kernel\) is 1-coercive, and that \(\func{f}{\X}{\Rinf}\) is proper, lsc, and \(\kernel\)-prox-bounded.
	Then, for any \(\lambda\in(0,\pb)\) the following hold:
	\begin{enumerate}
	\item
		\(\dom\env{f}=\dom\prox_{\lambda f}=\Y\).
	\item
		\(\func{\env{f}}{\Y}{\R}\) is continuous.
	\item \label{thm:mainprop:osc}%
		\(\ffunc{\prox_{\lambda f}}{\Y}{\X}\) is compact-valued and usc; in particular, it is locally bounded, osc, and \(\graph\prox_{\lambda f}\) is closed in \(\Y\times\R^n\).
	\end{enumerate}
\end{fact}
\begin{proof}
	For any \(y\in\Y\) and \(\lambda\in(0,\pb)\), it follows from \cref{thm:ext} that
	\(
		\psi_{y,\lambda}
	\coloneqq
		\tilde f + \lambda^{-1}\D({}\cdot{},y)
	\)
	is proper, lsc, and 1-coercive on \(\R^n\).
	The proof of \cite[Thm. 2.2]{kan2012moreau} can thus be invoked to obtain that \(\dom\prox_{\lambda f}=\dom\env{f}=\Y\), and that \(\prox_{\lambda f}\) is (nonempty- and) compact-valued.
	In turn, \cite[Cor. 2.2]{kan2012moreau} yields the sought continuity of \(\env{f}\).
	Finally, upper semicontinuity of \(\prox_{\lambda f}\) (on \(\Y\)) is shown in the proof of \cite[Thm. 2.6]{kan2012moreau}.
	\Cref{thm:usccpt} then yields local boundedness and closedness of \(\graph\prox_{\lambda f}\) in \(\Y\times\R^n\) (and in particular that \(\prox_{\lambda f}\) is osc).
\end{proof}

In the remainder of the subsection we make some important remarks regarding the statement and the ones in the cited reference.

\subsubsection{Rectified claims}
	By rigorously framing \(\env{f}\) and \(\prox_{\lambda f}\) as maps defined on \(\Y\), as opposed to \(\R^n\), our statement of \cref{thm:mainprop} rectifies some imprecise claims of the original reference \cite{kan2012moreau}.
	Indeed, note that \(\prox_{\lambda f}\) is usc only relative to \(\Y\), but not necessarily so if extended as a map on the entire space \(\R^n\); similarly, \(\graph\prox_{\lambda f}\) is closed only relative to \(\Y\times\R^n\), with no guarantee of closedness in \(\R^n\times\R^n\) as incorrectly claimed in \cite[Thm. 2.6 and Cor. 2.3]{kan2012moreau}.\footnote{%
		While the domain of definition of \(\prox_{\lambda f}\) is not explicitly mentioned, the definition of upper semicontinuity and \emph{closedness} (of the graph) is given for mappings between metric spaces (as opposed to \emph{subsets} thereof), see \cite[Def. 2.6 and 2.7]{kan2012moreau}.
	}
	A similar commentary extends to lower semicontinuity of \(\env{f}\) \cite[Thm. 2.2(iii)]{kan2012moreau}, \cite[Fact 2.2]{wang2022bregman}.
	In general, unless \(\kernel\) has full domain these properties are not guaranteed if \(\prox_{\lambda f}\) and \(\env{f}\) are viewed as objects defined on \(\R^n\), and this comment holds regardless of whether or not \(f\) admits lsc extensions on the whole space.

	\begin{example}\label{ex:0}%
		Let \(\kernel\) be \emph{any} dgf with \(\dom\kernel\subset\R^n\) and let \(f\equiv 0\).
		Viewed as an operator \(\R^n\rightrightarrows\R^n\), one has that \(\prox_{\lambda f}(y)\supseteq\set{y}\) if \(y\in\Y\) and is empty otherwise (the inclusion is an equality if \(\kernel\) is strictly convex).
		Since \(\Y=\interior\dom\kernel\) is not closed (in \(\R^n\)), at any boundary point \(\bar y\in\closure\Y\setminus\Y\) this operator is neither usc (by \cref{thm:domTclosed}) nor osc (consider \((x^k,y^k)=(y^k,y^k)\) in \cref{def:osc} with \(\Y\ni y^k\to\bar y\)).
		Similarly, \(\env{f}\) reduces to the indicator function of set \(\Y\), which is lsc on \(\Y\) but not on \(\R^n\).

		On the contrary, \(\prox_{\lambda f}\) is usc and osc and \(\env{f}\) is continuous \emph{relative to \(\Y\)}.
		No such specification is needed if \(\prox_{\lambda f}\) and \(\env{f}\) are considered as maps \emph{defined on \(\Y\)}, as we argue the natural practice should be.
	\end{example}

	These discrepancies are known, and the imprecise statements in this regard can be reasonably interpreted as minor typos or oversights.\footnote{%
		In \cite[Cor. 2.2]{kan2012moreau}, for instance, it is this time explicitly mentioned that \(\env{f}\) is continuous \emph{relative to \(\Y\)}, differently from the imprecise statement of \cite[Thm. 2.2(iii)]{kan2012moreau} were such restriction is not accounted for in the claims of lower semicontinuity.
	}
	We argue that the culprit lies in the conventional practice of treating \(\env{f}\) and \(\prox_{\lambda f}\) as maps defined on all of \(\R^n\).
	While widespread, this perspective introduces unnecessary complications and requires additional caveats to ensure correctness in otherwise straightforward statements.
	In contrast, when these objects are naturally defined on the open set \(\Y\)---the interior of the domain of the dgf \(\kernel\)---such ambiguities are automatically resolved.
	Notice further that our statement need not make any mention of the condition \(\dom f\cap\dom\kernel\neq\emptyset\), usually imposed in the literature to ensure that \(\env{f}\not\equiv\infty\).
	Instead, by regarding \(f\) as a function \emph{defined on \(\X\)}, mere properness of \(f\) automatically ensures this condition.

\subsubsection{A more general statement}
	The conventional approach of restricting attention to proper, lower semicontinuous functions \(\func{f}{\R^n}{\Rinf}\) may exclude an important and practically significant class of functions---those that are \emph{relatively weakly convex}, meaning they become convex when a suitable multiple of \(\kernel\) is added.
	A representative example is given in \cref{ex:ln}.
	Functions in this broad class possess highly desirable analytical properties: under mild conditions on \(\kernel\), one can show that \(\env{f}\circ\nabla\kernel* \) is continuously differentiable, and the proximal mapping \(\prox_{\lambda f}\) is continuous for sufficiently small \(\lambda\).

	\begin{example}\label{ex:ln}%
		Let \(\func{\kernel}{\R}{\Rinf}\) be defined as \(\kernel(x)=-\ln(x)\) for \(x\in\X=(0,\infty)\) and \(\kernel(x)=\infty\) for \(x\leq 0\).
		Then, \(\func{f}{\X}{\Rinf}\) given by \(f(x)=\ln(x)\) is proper, lsc, and \(\kernel\)-prox-bounded with \(\pb=1\).
		Although \(f\) does not admit a proper and lsc extension on \(\R\), \(\prox_{\lambda f}\) and \(\env{f}\) still enjoy the properties stated in \cref{thm:mainprop} for any \(\lambda\in(0,\pb)\).
	\end{example}

	This shortcoming is fully addressed in our convention, which imposes requirements on functions only on their ``natural'' set of definition \(\X\).

\subsubsection{The role of 1-coercivity}
	It is also important to remark the key role that 1-coercivity of \(\kernel\) has for \(\kernel\)-prox-boundedness to be meaningful, even when \(\kernel\) is Legendre with full domain.

	\begin{example}[importance of 1-coercivity in \(\kernel\)-prox-boundedness]%
		Consider \(\kernel=\exp\) with \(\X=\Y=\R\), which has \(\kernel*(\xi)=\xi\ln\xi-\xi\) for \(\xi\geq0\) (with the convention \(0\cdot\ln0=0\)).
		Let \(f(x)=x\).
		It is easy to see that \(f\) is \(\kernel\)-prox-bounded with \(\pb=\infty\), having
		\begin{align*}
			\env{f}(y)
		={} &
			\begin{cases}
				\lambda^{-1}\bigl[
					\kernel*(e^y)
					-
					\kernel*\bigl(e^y-\lambda\bigr)
				\bigr]
			&
				\text{if }
				e^y\geq \lambda
			\\
				-\infty
			&
				\text{otherwise,}
			\end{cases}
		\shortintertext{%
			and
		}
			\prox_{\lambda f}(y)
		={} &
			\begin{cases}
				\set{\ln(e^y-\lambda)}
			&
				\text{if }
				e^y>\lambda
			\\
				\emptyset
			&
				\text{otherwise.}
			\end{cases}
		\end{align*}
		Apparently, however, there exists no \(\lambda>0\) for which \(\env{f}\) is either proper or lsc (at \(y=\ln\lambda\)), nor for the proximal mapping to be locally bounded (at \(y=\ln\lambda)\) or with full domain.
	\end{example}

		\subsection{Alternative expressions}\label{sec:alternative}%
			This subsection lists some useful alternative expressions for the envelopes and proximal hull that involve convex conjugates and Euclidean forms.
Those involving convex conjugates extend those in \cite[Thm. 2.4]{kan2012moreau} and \cite[Prop. 2.14(ii)]{wang2022bregman} to account for the restricted definition \(\func{\env{f}}{\Y}{\Rinf}\) (as opposed to \(\func{\env{f}}{\R^n}{\Rinf}\)) that we adopt throughout, as well as waiving requirements on \(f\) or blanket assumptions of 1-coercivity on \(\kernel\); see \cite[Rem. 1]{kan2012moreau} and \cite[p. 1380]{wang2022bregman}.
The ones involving Euclidean expressions are derived as special cases of more general formulae, which, to the best of our knowledge, are novel.

\subsubsection{Expressions involving convex conjugates}
	\begin{lemma}\label{thm:left*}%
		Given \(\func{f}{\X}{\Rinf}\), extend it to \(\func{\tilde f}{\R^n}{\Rinf}\) by setting it to \(\infty\) on \(\R^n\setminus\X\) as in \cref{def:fext}.
		Then, for any \(\lambda>0\) the following hold:
		\begin{enumerate}
		\item \label{thm:envconj}%
			\(
				\lambda\env{f}
			=
				\bigl[\kernel*-(\lambda\tilde f+\kernel)^*\bigr]\circ\nabla\kernel
			\).
			In particular, when \(\kernel\) is of Legendre type,
			\(
				\lambda\env{f}\circ\nabla\kernel*
			=
				\bigl[\kernel*-(\lambda\tilde f+\kernel)^*\bigr]\restr_{\Y*}
			\)
			is locally Lipschitz continuous.

		\item \label{thm:proxconj}%
			\(
				\prox_{\lambda f}(y)
			=
				\bigl(\fsubdiff(\lambda\tilde f+\kernel)\bigr)^{-1}\circ\nabla\kernel
			\).

		\item \label{thm:hull}%
			\(
				\begin{array}[t]{@{}r@{}l@{}}
					\lambda\hull_{\lambda}{f}
				=
					\bigl((\lambda\tilde f+\kernel)^*+\indicator_{\range\nabla\kernel}\bigr)^*
					\restr_{\X}
					-
					\kernel
				\leq{} &
					\conv*(\lambda f+\kernel)
					-
					\kernel
				\\
				\leq{} &
					\conv(\lambda f+\kernel)
					-
					\kernel
				\leq
					\lambda f.
				\end{array}
			\)

			If \(\kernel\) is essentially smooth and 1-coercive, then the first inequality holds as equality.
			If instead \(\kernel\) is 1-coercive and \(f\) is \(\kernel\)-prox-bounded with \(\lambda<\pb\), then the second inequality holds as equality.
		\end{enumerate}
	\end{lemma}
	\begin{proof}~
		\begin{itemize}
		\item ``\ref{thm:envconj} and \ref{thm:proxconj}''
			It follows from \eqref{eq:DFY} that for any \(y\in\Y\) one has
			\begin{align*}
				\lambda\env{f}(y)
			={} &
				-\sup_{x\in\R^n}\set{
					\innprod{\nabla\kernel(y)}{x}
					-
					(\lambda\tilde f+\kernel)(x)
				}
				+
				\kernel*(\nabla\kernel(y))
			\end{align*}
			which is \(-(\lambda\tilde f+\kernel)^*(\nabla\kernel(y))+\kernel*(\nabla\kernel(y))\), as claimed.
			Notice that, since \(\nabla\kernel(y)\in\dom\lsubdiff\kernel*\subseteq\X*\), this expression is well determined (albeit possibly \(-\infty\)).
			The same reasoning for the set of minimizers, in combination with \eqref{eq:Fermat}, leads to the expression of \(\prox_{\lambda f}\).
			Finally, when \(\kernel\) is Legendre \(\env{f}\circ\nabla\kernel*\) is the difference of convex functions that are finite-valued on the nonempty open convex set \(\Y*\), and thus locally Lipschitz.

		\item ``\ref{thm:hull}''
			By direct computation, for any \(x\in\X\) one has
			\begin{align*}
				\lambda\hull_{\lambda}{f}(x)
			={} &
				-\lambda\env*{\bigl(-\env{f}\bigr)}(x)
			\\
			={} &
				\sup_{y\in\Y}\set{
					\lambda\env{f}(y)
					-
					\D(x,y)
				}
			\\
			={} &
				\sup_{y\in\Y}\inf*_{x'\in\X}\set{
					\lambda f(x')
					+
					\D(x',y)
					-
					\D(x,y)
				}
			\\
			={} &
				\sup_{y\in\Y}\inf*_{x'\in\R^n}\set{
					(\lambda\tilde f+\kernel)(x')
					+
					\innprod{\nabla\kernel(y)}{x-x'}
				}
				-
				\kernel(x)
			\\
			={} &
				\sup_{y\in\Y}\set{
					\innprod{\nabla\kernel(y)}{x}
					-
					\sup_{\mathclap{x'\in\R^n}}\set{
						\innprod{\nabla\kernel(y)}{x'}
						-
						(\lambda\tilde f+\kernel)(x')
					}
				}
				-
				\kernel(x)
			\\
			={} &
				\sup_{y\in\Y}\set{
					\innprod{\nabla\kernel(y)}{x}
					-
					(\lambda\tilde f+\kernel)^*\bigl(\nabla\kernel(y)\bigr)
				}
				-
				\kernel(x),
			\end{align*}
			which is the first identity in the statement.
			Note that the fourth identity above holds since \(\kernel\) is finite at \(y\in\Y\).
			In turn, the claimed inequalities follow from \cref{thm:cvxext} together with the fact that
			\(
				\bigl((\lambda\tilde f+\kernel)^*+\indicator_{\range\nabla\kernel}\bigr)^*
			\leq
				\bigl((\lambda\tilde f+\kernel)^*\bigr)^*
			\),
			see \cite[Prop. 13.16(ii)]{bauschke2017convex}.

			If \(\kernel\) is 1-coercive and \(\lambda<\pb\), then \(\lambda\tilde f+\kernel\) is 1-coercive by \cref{thm:ext:1co}, hence \(\conv(\lambda\tilde f+\kernel)=\conv*(\lambda\tilde f+\kernel)\) by \cref{thm:cvxext}.

			If instead \(\kernel\) is essentially smooth and 1-coercive, then \(\range\nabla\kernel=\R^n\) and thus
			\(
				\bigl((\lambda\tilde f+\kernel)^*+\indicator_{\range\nabla\kernel}\bigr)^*
				\restr_{\X}
			=
				(\lambda\tilde f+\kernel)^{**}
				\restr_{\X}
			=
				\conv*(\lambda f+\kernel)
			\)
			by \cref{thm:cvxext}.
		\qedhere
		\end{itemize}\let\qed\relax
	\end{proof}

	Neither assumption of Legendreness or 1-coercivity can be waived for the first inequality in \cref{thm:hull} to hold as equality, even if \(f\) is convex.
	Indeed, the following examples demonstrate how the proximal hull may fail to capture meaningful properties of \(f\) when either requirement is not met.

	\begin{example}[importance of essential smoothness for the \(\kernel\)-proximal hull]\label{ex:esssmooth}%
		Consider \(\kernel=\j+\indicator_{\X}\) where \(\X=[-1,1]\), which is strongly convex, 1-coercive, but not essentially smooth, and let \(f=\indicator_{\set{p}}\restr_{\X}\) for some \(p\in(-1,1)\).
		For any \(\lambda>0\) it is easy to see that \(\env{f}(y)=\frac{1}{2\lambda}(p-y)^2\) and
		\[
			\lambda\hull_{\lambda}{f}(x)
		=
			\sup_{y\in(-1,1)}\set{
				\tfrac{1}{2}(p-y)^2
				-
				\tfrac{1}{2}(x-y)^2
			}
		=
			\tfrac{1}{2}
			(p^2-x^2+2\abs{p-x}).
		\]
		For any \(x\in\X\) this agrees with the formula in \cref{thm:hull}, having
		\begin{align*}
			\bigl[(\lambda\tilde f+\kernel)^*+\indicator_{\range\nabla\kernel}\bigr]^*(x)
			-
			\kernel(x)
		={} &
			\bigl[p{}\cdot{}-\tfrac{1}{2}p^2+\indicator_{(-1,1)}\bigr]^*(x)
			-
			\tfrac{1}{2}x^2
		\\
		={} &
			\sup_{\abs{\xi}<1}\set{
				\xi(x-p)
			}
			+
			\tfrac{1}{2}p^2
			-
			\tfrac{1}{2}x^2,
		\end{align*}
		which does lower bound
		\(
			\conv*(\lambda f+\kernel)(x)-\kernel(x)
		=
			\lambda f(x)
		\),
		but agrees with it only at \(x=p\).
	\end{example}

	\begin{example}[importance of 1-coercivity for the \(\kernel\)-proximal hull]\label{ex:1co}%
		Consider \(\kernel=\exp\) for \(\X=\R\), which is essentially smooth (in fact, Legendre) but not 1-coercive.
		Let \(f=\indicator_{\set{0}}\), so that \(f=\tilde f\) and for all \(\lambda>0\) one has
		\(
			(\lambda\tilde f+\kernel)^*\equiv-1
		\).
		Then,
		\[
			\lambda\hull_{\lambda}{f}
		=
			\bigl(
				(\lambda\tilde f+\kernel)^*+\indicator_{\range\nabla\kernel}
			\bigr)^*
			-
			\kernel
		=
			\bigl(
				(\lambda\tilde f+\kernel)^*+\indicator_{\R_{++}}
			\bigr)^*
			-
			\kernel
		=
			1-\exp+\indicator_{\R_-}.
		\]
		As granted by \cref{thm:hull}, \(\lambda\hull_{\lambda}{f}\) furnishes a global lower bound to
		\(
			\conv*(\lambda f+\kernel)-\kernel(x)
		=
			\lambda f
		\),
		but agrees with it only on \(\R_+\).
	\end{example}

\subsubsection{Change of dgf}
	Interestingly, the left Bregman proximal map, Moreau envelope, and proximal hull can always be represented as Euclidean counterparts of an appropriately transformed function.
	A \emph{local} version of this correspondence was established in \cite[Thm. 3.8]{ahookhosh2021bregman}, where it enabled the derivation of differentiability and continuity properties via standard Euclidean results involving \emph{prox-regularity} \cite{poliquin1996proxregular}.\footnote{%
		These conclusions in \cite{ahookhosh2021bregman} required additional assumptions on the dgf \(\kernel\), such as local strong convexity and local Lipschitz continuity of \(\nabla\kernel\) on \(\interior\dom\kernel\).
		These limitations are overcome in the more refined framework of \cite{laude2020bregman}, where a notion of prox-regularity adapted to the Bregman setting is introduced.
	}

	The following lemma enables a \emph{global} reformulation of these Bregman objects.
	In fact, the Euclidean form is obtained as a special case of a more general result that allows for the \emph{change of reference function}, that is, replacing the original dgf \(\kernel\) with an alternative one, under suitable conditions.
	Similar expressions will be developed for the right Bregman operators in \cref{sec:changedgf*}.

	\begin{lemma}[change of dgf for left Bregman operators]\label{thm:chagedgf}%
		Let \(\func{\psi}{\R^n}{\Rinf}\) be a Legendre dgf such that \(\dom\psi\supseteq\dom\kernel\) and \(\range\nabla\psi\supseteq\range\nabla\kernel\).
		Given \(\func{f}{\X}{\Rinf}\) and \(\lambda>0\), define \(\func{F}{\dom\psi}{\Rinf}\) as
		\[
			F
		\coloneqq
			f+\tfrac{1}{\lambda}(\kernel-\psi)
			\text{ on \(\X\), and \(\infty\) on \(\dom\psi\setminus\X\).}
		\]
		Then, the following identities hold:
		\begin{enumerate}
		\item \label{thm:chagedgf:prox}%
			\(
				\prox_{\lambda f}(y)
			=
				\prox^\psi_{\lambda F}
				\circ
				\nabla\psi^*
				\circ
				\nabla\kernel(y)
			\)
			for all \(y\in\Y\).

		\item \label{thm:chagedgf:env}%
			\(
				\env{f}(y)
			=
				\bigl[
					\env_{\lambda}^\psi F
					\circ
					\nabla\psi^*
					+
					\tfrac{\kernel*-\psi^*}{\lambda}
				\bigr]
				\circ
				\nabla\kernel(y)
			\)
			for all \(y\in\Y\).
		\end{enumerate}
		If \(\kernel\) is Legendre and 1-coercive (and thus so is \(\psi\)), then the following also holds:%
		\begin{enumerate}[resume]
		\item \label{thm:chagedgf:hull}%
			\(
				\hull_{\lambda}{f}(x)
			=
				\bigl[
					\hull_{\lambda}^\psi F
					-
					\tfrac{\kernel-\psi}{\lambda}
				\bigr](x)
			\)
			for all \(x\in\X\).
		\end{enumerate}
	\end{lemma}
	\begin{proof}%
		~\begin{itemize}
		\item ``\ref{thm:chagedgf:prox} and \ref{thm:chagedgf:env}''
			By using \eqref{eq:DFY}, for any \(x\in\X\subseteq\dom\psi\) and \(y\in\Y\) one has that
			\[
				\D_\psi\bigl(x,\nabla\psi^*\circ\nabla\kernel(y)\bigr)
			=
				\psi(x)
				+
				\psi^*(\nabla\kernel(y))
				-
				\innprod{\nabla\kernel(y)}{x},
			\]
			with each term being finite by assumption, since \(\nabla\kernel(y)\in\range\nabla\kernel\subseteq\range\nabla\psi=\dom\nabla\psi^*\) by Legendreness.
			Therefore,
			\begin{align*}
				\D(x,y)
			={} &
				\kernel(x)
				+
				\kernel*(\nabla\kernel(y))
				-
				\innprod{\nabla\kernel(y)}{x}
			\\
			={} &
				(\kernel-\psi)(x)
				+
				(\kernel*-\psi^*)(\nabla\kernel(y))
				+
				\D_\psi\bigl(x,\nabla\psi^*\circ\nabla\kernel(y)\bigr).
			\numberthis\label{eq:Dpsi}
			\end{align*}
			After dividing by \(\lambda\) and adding \(f(x)\), the minimization with respect to \(x\) yields the claimed expressions for \(\env{f}(y)\) and \(\prox_{\lambda f}(y)\).

		\item ``\ref{thm:chagedgf:hull}''
			When \(\kernel\) is Legendre and 1-coercive, it follows from the assumptions on \(\psi\) that \(\range\nabla\kernel=\range\nabla\psi=\R^n\), and in particular that \(\psi\) too is 1-coercive.
			The formula from the \(\kernel\)-proximal hull follows from the fact that
			\begin{align*}
				\lambda\hull_{\lambda}f(x)
			={} &
				[\lambda\tilde f+\kernel]^{**}(x)-\kernel(x)
			\\
			={} &
				\left[
					\lambda\bigl(\tilde f+\tfrac{\kernel-\psi}{\lambda}\bigr)
					+
					\psi
				\right]^{**}\!\!(x)
				-
				\psi(x)
				-
				(\kernel-\psi)(x)
			\\
			={} &
				\lambda\hull_\lambda^\psi\left(
					\tilde f+\tfrac{\kernel-\psi}{\lambda}
				\right)(x)
				-
				(\kernel-\psi)(x),
			\end{align*}
			where the first and last identities owe to \cref{thm:hull} (since \(\indicator_{\range\nabla\kernel}=\indicator_{\R^n}\equiv0\)).
		\qedhere
		\end{itemize}\let\qed\relax
	\end{proof}

	\begin{corollary}[Euclidean form of left Bregman operators]\label{thm:Euclidean}%
		Given a function \(\func{f}{\X}{\Rinf}\), extend it to \(\func{\tilde f}{\R^n}{\Rinf}\) by setting it to \(\infty\) on \(\R^n\setminus\X\) as in \cref{def:fext}.
		Then, for any \(\lambda>0\) the following identities hold:
		\begin{enumerate}
		\item
			\(
				\prox_{\lambda f}
			=
				\Eprox_{\lambda\left(\tilde f+\frac{\kernel-\j}{\lambda}\right)}\circ\nabla\kernel
			\).
		\item
			\(
				\env{f}
			=
				\bigl[
					\Eenv_{\lambda}(\tilde f+\tfrac{\kernel-\j}{\lambda})
					+
					\tfrac{\kernel*-\j}{\lambda}
				\bigr]
				\circ
				\nabla\kernel
			\).
		\end{enumerate}
		If \(\kernel\) is Legendre and 1-coercive, then the following also holds:
		\begin{enumerate}[resume]
		\item
			\(
				\hull_{\lambda}{f}
			=
				\bigl[
					\Ehull_{\lambda}(\tilde f+\tfrac{\kernel-\j}{\lambda})
					-
					\tfrac{\kernel-\j}{\lambda}
				\bigr]\restr_{\X}
			\).
		\end{enumerate}
	\end{corollary}

	It should be noted that, despite the apparent \emph{resemblance} to the Euclidean case, this correspondence does \emph{not} trivialize the Bregman framework.
	The reformulated function \(\tilde f+\frac{\kernel-\j}{\lambda}\) depends intricately on both the dgf \(\kernel\) and the stepsize \(\lambda\), in a nested and non-separable way.
	As a result, classical Euclidean arguments cannot in general be directly transferred, and a genuinely Bregman-specific analysis of all the operators remains necessary.

		\subsection{The right proximal mapping and Moreau envelope}\label{sec:right}%
			While our treatment of the left Bregman-Moreau envelope began with the notion of prox-boundedness, our approach to the \emph{right} envelope follows an inverse route.
As we will see, a complete account of the consequences of \emph{right} prox-boundedness capturing and extending existing results revolves around properties of a particular \emph{epi-composite} mapping, whose detailed account requires some technical preliminaries.
For these reasons, we defer the formal definition of \emph{right \(\kernel\)-prox-boundedness} until later in the section, where we also establish its role in ensuring well-posedness of the right Bregman proximal mapping and the regularity of the associated envelope.

Instead, we start by providing a representation of the right envelope and proximal hull involving Fenchel conjugates that serves as the right counterpart of \cref{thm:left*}.
These expressions in the full generality of \(\kernel\) involve the notion of \emph{epi-composition} \cite[Eq. 1(17)]{rockafellar1998variational}, also known as \emph{infimal postcomposition} \cite[Def. 12.34]{bauschke2017convex}, \emph{marginal function} \cite[Eq. (3.24)]{auslender2002asymptotic}, or \emph{image function} \cite[Def. 5.1]{themelis2020douglas}.
We follow the terminology of \cite{rockafellar1998variational}, so that the \emph{epi-composition} of a function \(\func{g}{\Y}{\Rinf}\) by \(\func{\nabla\kernel}{\Y}{\R^n}\) is \(\func{\epicomp}{\R^n}{\Rinf}\) given by
\begin{equation}\label{eq:epicomp}
	\R^n\ni\xi
\mapsto
	\epicomp(\xi)
\coloneqq
	\inf_{y\in\Y}\set{g(y)}[\nabla\kernel(y)=\xi]
\end{equation}
(with \(\epicomp(\xi)=\infty\) whenever \(\xi\notin\range\nabla\kernel\)).\footnote{%
	Note that there is no need to involve the canonical extension \(\tilde g\) as in \cref{def:fext}, since \(g\) and \(\nabla\kernel\) have same domain of definition \(\Y\).
}
When \(\kernel\) is Legendre, \eqref{eq:epicomp} simplifies as \(\epicomp=g\circ\nabla\kernel*\) on \(\Y*\) and \(\infty\) elsewhere, and the expression of the right envelope in the following \cref{thm:env*conj} recovers the identity in \cite[Prop. 2.4(ii)]{bauschke2018regularizing}\footnote{%
	With a typo in the statement of \cite[Prop. 2.4(ii)]{bauschke2018regularizing}: \(\nabla f\) in place of \(\nabla f^*\) in the notation therein.
}
(up to exercising care in the domains of definition; see the technical discussion in \cref{sec:pb*}).
To the best of our knowledge, all the expressions in absence of Legendreness are instead novel.

\begin{lemma}\label{thm:right*}%
	Given \(\func{g}{\Y}{\Rinf}\), let \(\func{\epicomp}{\R^n}{\Rinf}\) be as in \eqref{eq:epicomp}.
	For any \(\lambda>0\) the following hold:
	\begin{enumerate}
	\item \label{thm:env*conj}%
		\(
			\lambda\env*{g}
		=
			\left[
				\kernel
				-
				\bigl(\kernel*+\lambda\epicomp\bigr)^*
			\right]\Restr_{\X}
		\).

	\item \label{thm:prox*conj}%
		\(
			\nabla\kernel\circ\prox*_{\lambda g}
		=
			\fsubdiff
			\bigl(
				\kernel*
				+
				\lambda\epicomp
			\bigr)^{-1}
		\).
		In particular, when \(\kernel\) is Legendre one has that
		\(
			\prox*_{\lambda g}
		=
			\nabla\kernel*
			\circ
			\fsubdiff
			\bigl(
				\kernel*
				+
				\lambda g\circ\nabla\kernel*
			\bigr)^{-1}
		\).

	\item \label{thm:hull*}%
		\(
			\begin{array}[t]{@{}r@{}l@{}}
				\lambda\hull*_{\lambda}{g}
			={} &
				\bigl(
					\bigl[
						(\kernel*+\lambda\epicomp)^*
						+
						\indicator_{\X}
					\bigr]^*
					-
					\kernel*
				\bigr)\circ\nabla\kernel
			\\
			\leq{} &
				\bigl(
					\conv*(\kernel*+\lambda\epicomp)
					-
					\kernel*
				\bigr)\circ\nabla\kernel
			\\
			\leq{} &
				\bigl(
					\conv(\kernel*+\lambda\epicomp)
					-
					\kernel*
				\bigr)\circ\nabla\kernel
			\\
			\leq{} &
				\lambda g.
			\end{array}
		\)

		If \(\dom\kernel=\R^n\) (i.e., \(\X=\Y=\R^n\)), then the first inequality is an equality.%
	\end{enumerate}
\end{lemma}
\begin{proof}~
	\begin{itemize}
	\item ``\ref{thm:env*conj} and \ref{thm:prox*conj}''
		By using \eqref{eq:DFY}, for any \(x\in\X\) one has
		\begin{align*}
			\lambda\env*{g}(x)
		={} &
			\inf_{y\in\Y}\set{
				\lambda g(y)
				+
				\kernel*(\nabla\kernel(y))
				-
				\innprod{\nabla\kernel(y)}{x}
			}
			+
			\kernel(x)
		\\
		={} &
			\inf_{\substack{y\in\Y,\xi\in\R^n\\\xi=\nabla\kernel(y)}}\set{
				\lambda g(y)
				+
				\kernel*(\xi)
				-
				\innprod{\xi}{x}
			}
			+
			\kernel(x)
		\\
		={} &
			\inf_{\xi\in\R^n}\set{
				\lambda\epicomp(\xi)
				+
				\kernel*(\xi)
				-
				\innprod{\xi}{x}
			}
			+
			\kernel(x),
		\numberthis\label{eq:env*:epicomp}
		\end{align*}
		which is
		\(
			\kernel(x)
			-
			\bigl(\kernel*+\lambda\epicomp\bigr)^*(x)
		\),
		as claimed.
		The same reasoning for the set of minimizers, in combination with \eqref{eq:Fermat}, leads to the expression of \(\nabla\kernel\circ\prox*_{\lambda g}\).%

	\item ``\ref{thm:hull*}''
		By direct computation, for any \(y\in\Y\) one has
		\begin{align*}
			\lambda\hull*_{\lambda}{g}(y)
		={} &
			-\lambda\env{\bigl(-\env*{g}\bigr)}(y)
		\\
		={} &
			\sup_{x\in\X}\set{
				\lambda\env*{g}(x)
				-
				\D(x,y)
			}
		\\
		={} &
			\sup_{x\in\X}
			\set{
				\kernel(x)
				-
				\bigl(\kernel*+\lambda\epicomp\bigr)^*(x)
				-
				\D(x,y)
			}
		\\
		={} &
			\sup_{x\in\X}
			\set{
				\innprod{\nabla\kernel(y)}{x}
				-
				\bigl(\kernel*+\lambda\epicomp\bigr)^*(x)
			}
			-
			\kernel*(\nabla\kernel(y))
		\\
		={} &
			\Bigl[\bigl(\kernel*+\lambda\epicomp\bigr)^*+\indicator_{\X}\Bigr]^*(\nabla\kernel(y))
			-
			\kernel*(\nabla\kernel(y))
		\end{align*}
		which is the first identity in the statement.
		The second and third inequalities follow from the same arguments as in the proof of \cref{thm:hull}.
		The last one owes to the fact that, for any \(y\in\Y\),
		\[
			\lambda\epicomp(\nabla\kernel(y))
		=
			\inf_{y':\nabla\kernel(y')=\nabla\kernel(y)} \lambda g(y')
		\leq
			\lambda g(y).
		\]
		When \(\X=\R^n\), it is clear that the first inequality holds as equality.
	\qedhere
	\end{itemize}\let\qed\relax
\end{proof}

\subsubsection{Change of dgf}\label{sec:changedgf*}%
	As with the left Bregman operators, the right counterparts can also be reformulated in terms of an alternative dgf \(\psi\).
	The choice \(\psi=\j\) yields representations involving standard Euclidean objects.

	\begin{lemma}[change of dgf for right Bregman operators]\label{thm:chagedgf*}%
		Let \(\func{\psi}{\R^n}{\Rinf}\) be a Legendre dgf such that \(\dom\psi\supseteq\dom\kernel\) and \(\range\nabla\psi\supseteq\range\nabla\kernel\).
		Given \(\func{g}{\Y}{\Rinf}\) and \(\lambda>0\), define \(\func{G}{\interior\dom\psi}{\Rinf}\) as
		\[
			G
		\coloneqq
			\left[
				\epicomp
				+
				\tfrac{1}{\lambda}(\kernel*-\psi^*)
			\right]
			\circ
			\nabla\psi.
		\]
		Then, the following identities hold:
		\begin{enumerate}
		\item \label{thm:chagedgf:prox*}%
			\(
				\nabla\kernel
				\circ
				\prox*_{\lambda g}(x)
			=
				\nabla\psi
				\circ
				\prox*^\psi_{\lambda G}(x)
			\)
			for all \(x\in\X\).
		\item \label{thm:chagedgf:env*}%
			\(
				\env*{g}(x)
			=
				\bigl[
					\env*_{\lambda}^\psi G
					+
					\tfrac{\kernel-\psi}{\lambda}
				\bigr](x)
			\)
			for all \(x\in\X\).
		\end{enumerate}
		If \(\kernel\) is Legendre and has full domain (and so does \(\psi\)), the following also holds:%
		\begin{enumerate}[resume]
		\item \label{thm:chagedgf:hull*}%
			\(
				\kernel*
				+
				\lambda
				\hull*_{\lambda}{g}
				\circ\nabla\kernel*
			=
				\psi^*
				+
				\lambda
				\hull*_{\lambda}^\psi G
				\circ\nabla\psi^*
			\)
			on \(\Y*\).
		\end{enumerate}
	\end{lemma}
	\begin{proof}~
		\begin{itemize}
		\item ``\ref{thm:chagedgf:prox*} and \ref{thm:chagedgf:env*}''
			We again use \eqref{eq:Dpsi} to express, for any \(x\in\X\), \(y\in\Y\), and with \(\eta=\nabla\kernel(y)\),
			\begin{align*}
				\D(x,y)
			={} &
				(\kernel-\psi)(x)
				+
				(\kernel*-\psi^*)(\eta)
				+
				\D_\psi\bigl(x,\nabla\psi^*(\eta)\bigr)
			\\
			={} &
				(\kernel-\psi)(x)
				+
				(\kernel*-\psi^*)\circ\nabla\psi\bigl(\nabla\psi^*(\eta)\bigr)
				+
				\D_\psi\bigl(x,\nabla\psi^*(\eta)\bigr).
			\end{align*}
			After dividing by \(\lambda\) and adding \(g(y)\), we compute
			{%
				\thickmuskip=.5\thickmuskip
				\medmuskip=.5\medmuskip
				\thinmuskip=.5\thinmuskip
				\begin{align*}
					\nabla\kernel\circ\prox*_{\lambda g}(x)
				={} &
					\nabla\kernel
					\circ
					\argmin_{\mathclap{\substack{y\in\Y}}}\set{
						g(y)
						+
						\tfrac{\kernel*-\psi^*}{\lambda}\circ\nabla\psi\bigl(\nabla\psi^*(\eta)\bigr)
						+
						\tfrac{1}{\lambda}
						\D_\psi(x,\nabla\psi^*(\eta))
					}
				\\
				={} &
					\argmin_{\eta\in\range\nabla\kernel}\biggl\{
						\overbracket[0.5pt]{
							\smash{
								\underbracket[0.5pt]{
									\min_{y\in\Y}\set{g(y)}[\nabla\kernel(y)=\eta]
								}_{\epicomp(\eta)}
							}
							+
							\tfrac{\kernel*-\psi^*}{\lambda}\circ\nabla\psi\bigl(\nabla\psi^*(\eta)\bigr)
						}^{G(\nabla\psi^*(\eta))}
				\\[-0.25cm]
				&
					\hspace*{6cm}
						+
						\tfrac{1}{\lambda}
						\D_\psi(x,\nabla\psi^*(\eta))
					\biggr\}
				\\
				={} &
					\argmin_{\eta\in\range\nabla\kernel}\set{
						G(\nabla\psi^*(\eta))
						+
						\tfrac{1}{\lambda}
						\D_\psi(x,\nabla\psi^*(\eta))
					}
				\\
				={} &
					\argmin_{\eta'\in\interior\dom\psi^*}\set{
						G(\nabla\psi^*(\eta'))
						+
						\tfrac{1}{\lambda}
						\D_\psi(x,\nabla\psi^*(\eta'))
					}
				\\
				={} &
					\nabla\psi
					\circ
					\argmin_{y'\in\interior\dom\psi}\set{
						G(y')
						+
						\tfrac{1}{\lambda}
						\D_\psi(x,y')
					},
				\end{align*}
			}%
			which yields the claimed relation among the proximal mappings.
			Here, the fourth identity uses the fact that \(\range\nabla\kernel\subseteq\range\nabla\psi=\interior\dom\psi^*\), together with the fact that \(\epicomp(\eta')=\infty\) for \(\eta'\in\interior\dom\psi^*\setminus\range\nabla\kernel\).
			A similar reasoning for the infimum proves the claimed relation among the envelopes as well.

		\item ``\ref{thm:chagedgf:hull*}''
			When \(\kernel\) is Legendre and has full domain, the formula from the \(\kernel\)-proximal hull follows from the fact that \(\epicomp=g\circ\nabla\kernel*\) on \(\range\nabla\kernel=\Y*\).
			Indeed, for any \(\eta\in\Y^*\) one has that
			\begin{align*}
				\lambda\hull*_{\lambda}g\circ\nabla\kernel*(\eta)
			={} &
				\bigl[
					(\kernel*+\lambda g\circ\nabla\kernel*)^{**}-\kernel*
				\bigr]
				(\eta)
			\\
			={} &
				\Bigl[
					\bigl(
						\psi^*
						+
						\lambda\bigl(
							g\circ\nabla\kernel*
							+
							\tfrac{\kernel*-\psi^*}{\lambda}
						\bigr)
					\bigr)^{**}
					-
					\psi^*
				\Bigr]
				(\eta)
				-
				(\kernel*-\psi^*)(\eta)
			\\
			={} &
				\bigl[
					(\psi^*+\lambda G\circ\nabla\psi^*)^{**}-\psi^*
				\bigr]
				(\eta)
				-
				(\kernel*-\psi^*)
				(\eta)
			\\
			={} &
				\lambda\hull*^\psi{G}\circ\nabla\psi^*(\eta)
				-
				(\kernel*-\psi^*)(\eta)
			\end{align*}
			where the first and last identities owe to \cref{thm:hull*}, the second one follows by adding and subtracting the (finite) quantity \(\psi^*(\eta)\) twice, and the third one uses the definition of \(G\).
		\qedhere
		\end{itemize}\let\qed\relax
	\end{proof}

	\begin{corollary}[Euclidean form of right Bregman operators]\label{thm:Euclidean*}%
		Given a function \(\func{g}{\Y}{\Rinf}\) and \(\lambda>0\), define \(\func{G}{\R^n}{\Rinf}\) as
		\[
			G
		\coloneqq
			\epicomp
			+
			\tfrac{1}{\lambda}(\kernel*-\j).
		\]
		Then, the following identities hold:
		\begin{enumerate}
		\item
			\(
				\nabla\kernel
				\circ
				\prox*_{\lambda g}(x)
			=
				\Eprox_{\lambda G}(x)
			\)
			for all \(x\in\X\).
		\item
			\(
				\env*{g}(x)
			=
				\bigl[
					\Eenv_{\lambda}G
					+
					\tfrac{\kernel-\j}{\lambda}
				\bigr](x)
			\)
			for all \(x\in\X\).
		\end{enumerate}
		If \(\kernel\) is Legendre and has full domain, then the following also holds:
		\begin{enumerate}[resume]
		\item
			\(
				\hull*_{\lambda}{g}
			=
				\left[
					\Ehull_{\lambda}G
					-
					\tfrac{\kernel*-\j}{\lambda}
				\right]
				\circ\nabla\kernel
			\).
		\end{enumerate}
	\end{corollary}

			\subsubsection{Right prox-boundedness}\label{sec:pb*}%
				We are now ready to elaborate on the \emph{right} counterpart of \(\kernel\)-prox-boundedness, following \cite[Def. 3.4]{laude2020bregman}.

\begin{definition}[Right \(\kernel\)-prox-boundedness {\cite[Def. 3.4]{laude2020bregman}}]\label{def:PB*}%
	We say that a function \(\func{g}{\Y}{\Rinf}\) is \emph{right \(\kernel\)-prox-bounded} if there exist \(\lambda>0\) and \(x\in\X\) such that \(\env*{g}(x)>-\infty\).
	The supremum \(\pb*\) of all such \(\lambda\) is the \emph{right \(\kernel\)-prox-boundedness threshold} of \(g\).
\end{definition}

By comparing the identity in \cref{thm:env*conj} with the left counterpart of \cref{thm:envconj}, it may seem natural to identify the right envelope \(\env*{g}\) with the left envelope \(\env^{\kernel*}{(g\circ\nabla\kernel*)}\circ\nabla\kernel\) whenever \(\kernel\) is Legendre.
However, such expression and any conclusion that follows must be made with care, since \(g\circ\nabla\kernel*\) must be a function defined on \(\X*\), as opposed to \(\Y*\), and additional requirements may be needed to ensure lower semicontinuity.
This subtlety was initially overlooked in the original reference \cite{laude2020bregman}, leading to imprecise conclusions, as illustrated in \cref{ex:env*}.
The issue was later addressed and corrected in the doctoral thesis \cite{laude2021lower}, where sufficient conditions on \(\kernel\) and \(g\) were identified to ensure the required lower semicontinuity.
Specifically, assuming \(\dom\kernel = \R^n\) (the necessary right-hand analogue of 1-coercivity), Legendreness of \(\kernel\) and either coercivity of \(g\) or 1-coercivity of \(\kernel\) were shown to suffice.

In this work, we confirm and generalize those findings by recasting them in terms of the lower semicontinuity of the epi-composition \(\epicomp\).
In fact, our framework based on restricted domains allows us to relax the assumptions: it suffices for \(\epicomp\) to be lower semicontinuous merely on \(\X*\), not necessarily on the whole \(\R^n\).
Under Legendreness, this condition is automatically satisfied when \(\kernel*\) has open domain (and \(g\) is lsc).
Beyond this, we further extend the analysis to cover cases where Legendreness is not assumed.

\begin{theorem}\label{thm:mainprop*}%
	Suppose that \(\dom\kernel=\R^n\) (i.e., \(\X=\Y=\R^n\)), and let \(\func{g}{\Y}{\Rinf}\) be proper and right \(\kernel\)-prox-bounded with threshold \(\pb*\).
	Then, for every \(\lambda\in(0,\pb*)\), and with \(\epicomp\) as in \eqref{eq:epicomp}, the following hold:
	\begin{enumerate}
	\item \label{thm:mainprop*:C0}%
		\(\func{\env*{g}}{\X}{\R}\) is locally Lipschitz continuous.
	\item \label{thm:mainprop*:pb*}%
		If \(\kernel\) is Legendre, then \(\epicomp\restr_{\X*}\) is left \(\kernel*\)-prox-bounded with same threshold \(\pb_{\epicomp}^{\kernel*}=\pb*\).
	\item \label{thm:mainprop*:epicomp}%
		\(\lambda\epicomp+\kernel*\) is 1-coercive.
		It is also lsc (on \(\R^n\)) provided that \(\epicomp\restr_{\X*}\) is lsc.
		For \(\epicomp\restr_{\X*}\) to be lsc it is sufficient that \(g\) be lsc and that one of the following conditions is satisfied:
		\begin{itemize}[nosep]
		\item
			either \(\kernel\) is Legendre and \(\X*\) is open,
		\item
			or \(g\) is coercive,
		\item
			or \(\kernel\) is 1-coercive.
		\end{itemize}
	\item \label{thm:mainprop*:osc}%
		\(\ffunc{\nabla\kernel\circ\prox*_{\lambda g}}{\X}{\R^n}\) is locally bounded.
		If, additionally, \(\epicomp\restr_{\X*}\) is lsc, then \(\nabla\kernel\circ\prox*_{\lambda g}\) is also osc, usc, and nonempty- and compact-valued.
	\end{enumerate}
\end{theorem}
\begin{proof}~
	\begin{itemize}
	\item ``\ref{thm:mainprop*:C0}''
		For a fixed \(\lambda\in(0,\pb*)\), let \(\lambda_0\in(\lambda,\pb*)\) and \(x^0\in\X\) be such that \(\env*_{\lambda_0}{g}(x^0)>-\infty\).
		By properness of \(g\), necessarily \(\env*_{\lambda_0}{g}(x^0)\in\R\).
		For every \(x\in\X\) we have
		\begin{align*}
		&
			\env*{g}(x)
		\\
		={} &
			\inf_{y\in\Y}\Bigl\{
				g(y)
				+
				\tfrac{1}{\lambda_0}\D(x^0,y)
				+
				\left(\tfrac{1}{\lambda}-\tfrac{1}{\lambda_0}\right)
				\D(x,y)
		\\
		&
		\qquad\quad
				+
				\tfrac{1}{\lambda_0}
				\Bigl(\kernel(x)-\kernel(x^0)-\innprod{\nabla\kernel(y)}{x-x^0}\Bigr)
			\Bigr\}
		\\
		\geq{} &
			\env*_{\lambda_0}{g}(x^0)
			+
			\tfrac{\kernel(x)-\kernel(x^0)}{\lambda_0}
			+
			\inf_{y\in\Y}\set{
				\tfrac{\lambda_0-\lambda}{\lambda_0\lambda}
				\D(x,y)
				-
				\tfrac{1}{\lambda_0}
				\innprod{\nabla\kernel(y)}{x-x^0}
			}
		\\
		={} &
			\env*_{\lambda_0}{g}(x^0)
			+
			\tfrac{\kernel(x)-\kernel(x^0)}{\lambda_0}
		\\
		&
			+
			\inf_{y\in\Y}\set{
				\tfrac{\lambda_0-\lambda}{\lambda_0\lambda}
				\left[\kernel(x)+\kernel*(\nabla\kernel(y))-\innprod{\nabla\kernel(y)}{x}\right]
				-
				\tfrac{1}{\lambda_0}
				\innprod{\nabla\kernel(y)}{x-x^0}
			}
		\\
		={} &
			\env*_{\lambda_0}{g}(x^0)
			+
			\tfrac{1}{\lambda}\kernel(x)
			-
			\tfrac{1}{\lambda_0}\kernel(x^0)
			+
			\tfrac{\lambda_0-\lambda}{\lambda_0\lambda}
			\inf_{\xi\in\range\nabla\kernel}\set{
				\kernel*(\xi)
				-
				\innprod{\xi}{
					\tfrac{\lambda_0x-\lambda x^0}{\lambda_0-\lambda}
				}
			}.
		\end{align*}
		Real valuedness of \(\kernel\) implies that \(\kernel*\) is 1-coercive, and as a consequence the infimum above, which is lower bounded by the infimum over \(\R^n\), is finite.
		From the arbitrariness of \(x\in\X\) we conclude that \(\env*{g}>-\infty\), which combined with properness of \(g\) yields that \(\env*{g}\) is everywhere real valued.

		In turn (recall that \(\X=\R^n\) by assumption), the expression in \cref{thm:env*conj} implies that \(\env*{g}=\kernel-(\kernel*+\lambda\epicomp)^*\) is the difference of two convex and real-valued functions, which are thus both locally Lipschitz on \(\R^n\).
		As a result, \(\env*{g}\) itself must be locally Lipschitz continuous.

	\item
		``\ref{thm:mainprop*:pb*}''
		The assumption of Legendreness ensures that \(\kernel*\) qualifies as a dgf, since it is differentiable on \(\interior\dom\kernel*\).
		The claim then follows from \eqref{eq:env*:epicomp}, by observing that the minimization there can be restricted over \(\xi\in\X*\)---in fact, even over \(\xi\in\interior\X*\)---for the function being minimized is \(\infty\) elsewhere.

	\item
		``\ref{thm:mainprop*:epicomp}''
		As shown in the proof of assertion \ref{thm:mainprop*:C0}, \((\kernel*+\lambda\epicomp)^*\) is everywhere real-valued, and thus \(\kernel*+\lambda\epicomp\) is 1-coercive.
		If \(\epicomp\restr_{\X*}\) is lsc, we can show lower semicontinuity of \(\lambda\epicomp+\kernel*\) on \(\R^n\) for \(\lambda<\pb*\) by patterning the arguments in the proof of \cref{thm:ext:lsc}.
		Indeed, take \(\lambda_0\in(\lambda,\pb*)\) and \(x^0\in\X\) such that \(\env*{g}(x^0)>-\infty\).
		We will show lower semicontinuity of
		\begin{align*}
			\psi_\lambda(\xi)
		\coloneqq{} &
			\epicomp(\xi)
			+
			\tfrac{1}{\lambda}
			\left(
				\kernel*(\xi)
				-
				\innprod{\xi}{x^0}
			\right)
		\\
		\geq{} &
			\env*_{\lambda_0}(x^0)
			+
			\left(
				\tfrac{1}{\lambda}
				-
				\tfrac{1}{\lambda_0}
			\right)
			\left(
				\kernel*(\xi)
				-
				\innprod{\xi}{x^0}
			\right),
		\end{align*}
		where the inequality follows from \eqref{eq:env*:epicomp}.
		Let \(\alpha\in\R\) and \(\seq{\xi^k}\) be a sequence converging to a point \(\bar\xi\) and such that \(\psi_{\lambda}(\xi^k)\leq\alpha\) for all \(k\).
		To arrive to a contradiction, suppose that \(\psi_{\lambda}(\bar\xi)>\alpha\).
		Since \(\psi_{\lambda}\) is lsc on \(\X*\) by assumption, one must have that \(\bar\xi\notin\X*\), that is, \(\kernel*(\bar\xi)=\infty\).
		On the other hand, for every \(k\) it follows from the inequality in the last display that
		\[
			\alpha
		\geq
			\env*_{\lambda_0}(x^0)
			+
			\left(
				\tfrac{1}{\lambda}
				-
				\tfrac{1}{\lambda_0}
			\right)
			\left(
				\kernel*(\xi^k)
				-
				\innprod{\xi^k}{x^0}
			\right)
		\to
			\infty
		\quad
			\text{as }k\to\infty
		\]
		by lower semicontinuity of \(\kernel*\), a contradiction.
		Thus, \(\lambda\epicomp+\kernel*=\lambda\psi_\lambda+\innprod{\xi}{x^0}\) is lsc on \(\R^n\).

		We next assume that \(g\) is lsc, and prove that each of the conditions listed in the statement suffices for lower semicontinuity of \(\epicomp\restr_{\X*}\).

		If \(\kernel\) is Legendre and \(\X*\) is open, the claim is obvious from continuity of \(\nabla\kernel\) and the fact that \(\range\nabla\kernel=\interior\X*\).

		As for the other two conditions, note that when \(\kernel\) is 1-coercive, \(\kernel*\) is finite-valued and its subdifferential is thus locally bounded.
		As such, when either \(g\) is level bounded or \(\kernel\) is 1-coercive, the set
		\(
			\set{y\in\Y}[g(y)\leq\alpha,~\norm{\nabla\kernel(y)-u}\leq\varepsilon]
		\)
		is bounded for any \(\alpha\in\R\), \(u\in\R^n\), and \(\varepsilon>0\), which by \cite[Prop. 1.32]{rockafellar1998variational} yields the claimed lower semicontinuity of \(\epicomp\) on \(\R^n\).

	\item
		``\ref{thm:mainprop*:osc}''
		By combining the expression in \cref{thm:prox*conj} with \eqref{eq:Fermat} we obtain that
		\[
			\nabla\kernel\circ\prox*_{\lambda g}(x)
		=
			\argmin_{\xi\in\R^n}\set{\lambda\epicomp(\xi)+\kernel*(\xi)-\innprod{\xi}{x}}.
		\]
		Note that 1-coercivity of \(\lambda\epicomp+\kernel*\) shown in assertion \ref{thm:mainprop*:epicomp} implies that the function being minimized is level bounded in \(\xi\) locally uniformly in \(x\), in the sense of \cite[Def. 1.16]{rockafellar1998variational}.
		When \(\epicomp\restr_{\X*}\) is lsc, that function is lsc on the whole space \(\R^n\times\R^n\) as it follows from assertion \ref{thm:mainprop*:epicomp}, and the claimed local boundedness and nonempty- and compact-valuedness follow from \cite[Thm. 1.17]{rockafellar1998variational}.
		Moreover, for any \(\xi\) notice that the section
		\(
			x
		\mapsto
			\lambda\epicomp(\xi)+\kernel*(\xi)-\innprod{\xi}{x}
		\)
		is continuous, which appealing again to \cite[Thm. 1.17]{rockafellar1998variational} yields the claimed outer semicontinuity of the corresponding optimal-set mapping \(\nabla\kernel\circ\prox*_{\lambda g}(x)\).
		In turn, upper semicontinuity follows from \cref{thm:oscusc:Yclosed}.
	\qedhere
	\end{itemize}\let\qed\relax
\end{proof}

\begin{example}[importance of lower semicontinuity of \(\epicomp\)]\label{ex:env*}%
	Real valuedness and Legendreness of \(\kernel\) are not enough to guarantee nonemptiness of the right proximal mapping for any \(\lambda\in(0,\pb*)\), as imprecisely stated in \cite[Lem. 3.5]{laude2020bregman}.
	To see this, with \(\X=\Y=\R\) let
	\[
		\kernel(x)
	=
		\begin{cases}
			x^2-2x+3 & \text{if }x<1
		\\
			x+\frac{1}{x} & \text{if }x\geq1,
		\end{cases}
	\]
	and \(\func{g}{\Y}{\Rinf}\) be given by \(g(y)=\frac{1}{y}\) for \(y>0\) and \(\infty\) otherwise.
	Clearly, both \(\kernel\) and \(g\) are convex and Legendre, \(\kernel\) is coercive (but \emph{not 1-}coercive) and has also full domain.
	Moreover, \(\env*{g}>0\) for any \(\lambda>0\), and in particular \(\pb*=\infty\).
	For \(\bar x=1\) and \(y\geq1\), one has that
	\[
		g(y)
		+
		\tfrac{1}{\lambda}
		\D(\bar x,y)
	=
		\tfrac{(y-1)^2}{\lambda y^2}
		+
		\tfrac{1}{y},
	\]
	and it can easily be verified by computing the derivative that this function is strictly decreasing in \(y\) whenever \(\lambda\geq2\).
	The same is true also for \(y<1\), since \(\D(\bar x,y)=\norm{\bar x-y}^2\) in this case, altogether implying that the set of minimizers \(\prox*_{\lambda g}(\bar x)\) is empty for all \(\lambda\geq2\).
	Not surprisingly, the epi-composition
	\[
		\epicomp(\xi)
	=
		\begin{cases}
			\frac{2}{\xi+2} & \text{if }\, -2<\xi<0\\
			\sqrt{1-\xi} & \text{if }\, 0\leq\xi<1\\
			\infty & \text{if }\, \xi\leq-2 \vee \xi\geq1
		\end{cases}
	\]
	is not lower semicontinuous (at \(\xi=1\in\dom\kernel*\)), and indeed the sufficient conditions of \cref{thm:mainprop*:epicomp} are not satisfied: neither is \(g\) coercive, nor is \(\kernel\) 1-coercive, nor is \(\dom\kernel*=(-\infty,1]\) open.
\end{example}

In conclusion, leveraging on \cref{thm:mainprop*:epicomp} we may provide a refined version of \cref{thm:hull} under right \(\kernel\)-prox-boundedness.

\begin{lemma}\label{thm:hull*+}%
	Suppose that \(\dom\kernel=\R^n\) (i.e., \(\X=\Y=\R^n\)), and let \(\func{g}{\Y}{\Rinf}\) be proper, lsc, and right \(\kernel\)-prox-bounded.
	Then, for any \(\lambda\in(0,\pb*)\) it holds that
	\[
		\begin{array}[t]{@{}r@{}l@{}}
			\lambda\hull*_{\lambda}{g}
		={} &
			\bigl(
				\conv*(\kernel*+\lambda\epicomp)
				-
				\kernel*
			\bigr)\circ\nabla\kernel
		\\
		={} &
			\bigl(
				\conv(\kernel*+\lambda\epicomp)
				-
				\kernel*
			\bigr)\circ\nabla\kernel
		\\
		\leq{} &
			\lambda g.
		\end{array}
	\]
\end{lemma}
\begin{proof}
	The first equality and the inequality were shown in \cref{thm:hull*}.
	The second equality follows from \cref{thm:cvxext}, since \(\kernel*+\lambda\epicomp\) is 1-coercive by \cref{thm:mainprop*:epicomp}.
\end{proof}

	\section{A \texorpdfstring{\(\Phikernel\)}{Phi}-conjugacy perspective}\label{sec:Phiconj}
		Akin to the analysis in \cite{laude2023dualities}, we frame the objects under investigation in this paper under the lens of generalized \(\Phikernel\)-conjugacy; see for example \cite[Sec. 2.6.2]{figalli2023invitation}, \cite[Sec. 7.2.4]{rubinov2013abstract}, and \cite[Chap. 11]{rockafellar1998variational}.
This perspective will corroborate our claims about \(\X\) and \(\Y\) being their ``natural'' (co)domains.
Beyond this, we believe that this theory is extraordinarily appealing in its ability to coincisely explain many key relations.
We repropose some known preliminary results, of which we detail the simple proofs for the sake of keeping this section self-contained.
For further details and developments we refer the interested reader to the more comprehensive work \cite{laude2023dualities} by Laude et al., and to \cite[Sec. 7]{rubinov2013abstract} for a more general overview on \(\Phikernel\)-convexity.

		\subsection{Main definitions and Fenchel \texorpdfstring{\(\Phikernel\)}{Phi}-duality}%
			In this subsection, let \(X\) and \(Y\) be \emph{any} nonempty sets, and let \(\func{\Phikernel}{X\times Y}{\R}\) be \emph{any} real-valued function, which we shall refer to as a \emph{coupling} between \(X\) and \(Y\).
Correspondingly, let \(\func{\Phikernel*}{Y\times X}{\R}\) be defined as
\begin{equation}
	\Phikernel*(y,x)=\Phikernel(x,y)
\qquad
	\forall (y,x)\in Y\times X.
\end{equation}
As we will see, the Bregman objects under investigation will be obtained with the specific choice of \(\Phikernel=-\frac{1}{\lambda}\D\).
Before specializing the results to this choice, we provide the general picture holding for arbitrary sets and coupling functions.
This degree of generality will ultimately be useful in \cref{thm:Bcoco} to obtain a new characterization of \emph{relative smoothness}.

We begin with the definition of \(\Phikernel\)-conjugate.

\begin{definition}[\(\Phikernel\)-conjugate]\label{def:conj}%
	The \emph{\(\Phikernel\)-conjugate} of \(\func{f}{X}{\Rinf}\) is \(\func{\conj{f}}{Y}{\Rinf}\) defined by
	\begin{equation}\label{eq:conj}
		\conj{f}(y)\coloneqq\sup_{x\in X}\set{\Phikernel(x,y)-f(x)}.
	\end{equation}
\end{definition}

Since the marginalization is performed with respect to the first argument \(x\in X\), \(\conj f\) can be thought of as a \emph{``left''} \(\Phikernel\)-conjugate.
A \emph{``right''} counterpart can similarly be done for functions \(\func{g}{Y}{\Rinf}\); this operation can still be represented as a ``left'' conjugation as in \cref{def:conj} by using \(\Phikernel*\) as coupling, as it corresponds to \(\Phikernel\) with ``flipped'' arguments.
Hence, the \emph{\(\Phikernel*\)-conjugate} (or \emph{``right'' \(\Phikernel\)-conjugate}) of \(\func{g}{Y}{\Rinf}\) is \(\func{\conj*{g}}{X}{\Rinf}\) defined by
\[
	\conj*{g}(x)
=
	\sup_{y\in Y}\set{\Phikernel(x,y)-g(y)},
\]
namely with marginalization with respect to the \emph{second} argument of \(\Phikernel\).
When \(\Phikernel\) is symmetric, in the sense that (\(X=Y\) and) \(\Phikernel(x,y)=\Phikernel(y,x)\), one has that \(\Phikernel*=\Phikernel\) and this distinction is thus unnecessary, and similarly the introduction of a dedicated notation for the ``flipped'' function \(\Phikernel*\) becomes superfluous.\footnote{%
	In fact, even in absence of symmetry it is a common practice in the literature to use a same superscript, say, \(\Phikernel\), both for left and right conjugacy operations, leaving the distinction based on the context.
	We however prefer to stick to a more rigorous and unambiguous notation that explicitly clarifies this distinction, akin to the conventions adopted in, e.g., \cite{martinezlegaz2005generalized,cabot2017envelopes}.
}
This is the case of the usual Fenchel conjugacy operation, which is recovered for \(X=Y=\R^n\) with \(\Phikernel=\Phikernel*=\func{*}{X\times Y}{\R}\) denoting the inner product \(*(x,y)=\innprod{x}{y}\).

Other objects such as the biconjugate and subdifferentials, defined next, are similarly recovered with this specific choice of coupling function.

\begin{definition}[\(\Phikernel\)-biconjugate]%
	The \emph{\(\Phikernel\)-biconjugate} of \(\func{f}{X}{\Rinf}\) is \(\biconj{f}\coloneqq\conj*{(\conj{f})}\), namely \(\func{\biconj{f}}{X}{\Rinf}\) defined by
	\[
		\biconj{f}(x)\coloneqq\sup_{y\in Y}\set{\Phikernel(x,y)-\conj{f}(y)}.
	\]
\end{definition}

\begin{definition}[\(\Phikernel\)-subgradients]\label{def:subdiff}%
	We say that \(\bar y\in Y\) is a \emph{\(\Phikernel\)-subgradient} of \(\func{f}{X}{\Rinf}\) at \(\bar x\in\dom f\), denoted by \(\bar y\in\subdiff f(\bar x)\), if
	\begin{gather*}
		f(x)\geq f(\bar x)+\Phikernel(x,\bar y)-\Phikernel(\bar x,\bar y)
	\quad
		\forall x\in X
	\shortintertext{or, equivalently, if}
	\numberthis\label{eq:subdiff}
		\bar x
	\in
		\argmax_{x\in X}\set{\Phikernel(x,\bar y)-f(x)}
	=
		\argmin_{x\in X}\set{f(x)-\Phikernel(x,\bar y)}.
	\end{gather*}
\end{definition}

As in the case of the \(\Phikernel*\)-conjugate, also \(\Phikernel*\)-biconjugate and \(\Phikernel*\)-subgradients of functions \(Y\to\Rinf\) can be expressed as ``right'' equivalents of the \(\Phikernel\)-counterparts.
That is, the \emph{\(\Phikernel*\)-biconjugate} (or \emph{``right'' \(\Phikernel\)-biconjugate}) of \(\func{g}{Y}{\Rinf}\) is \(\biconj*{g}\coloneqq\conj{(\conj*{g})}\), namely \(\func{\biconj*{g}}{Y}{\Rinf}\) defined by
\[
	\biconj*{g}(y)\coloneqq\sup_{x\in X}\set{\Phikernel(x,y)-\conj*{g}(x)},
\]
whereas \(\bar x\in X\) is a \emph{\(\Phikernel*\)-subgradient} (or \emph{``right'' \(\Phikernel\)-subgradient}) of \(g\) at \(\bar y\in\dom g\), denoted by \(\bar x\in\subdiff* g(\bar y)\), if
\begin{gather*}
	g(y)\geq g(\bar y)+\Phikernel(\bar x,y)-\Phikernel(\bar x,\bar y)
\quad
	\forall y\in Y
\shortintertext{or, equivalently, if}
	\bar y
\in
	\argmax_{y\in Y}\set{\Phikernel(\bar x,y)-g(y)}
=
	\argmin_{y\in Y}\set{g(y)-\Phikernel(\bar x,y)}.
\end{gather*}

\begin{fact}[{\cite[Lem. 2.4]{laude2023dualities}}]%
	The following hold for any proper \(\func{f}{X}{\Rinf}\):
	\begin{enumerate}
	\item \label{thm:FY:ineq}%
		\(f(x)+\conj{f}(y)\geq\Phikernel(x,y)\) for any \((x,y)\in X\times Y\).
	\item \label{thm:FY:biconj<=}%
		\(\biconj f\leq f\).
	\item \label{thm:FY:triconj}%
		\(\conj{f}=\conj{\biconj f}>-\infty\).
	\end{enumerate}
\end{fact}
\begin{proof}
	We start by observing that properness of \(f\) ensures that \(\conj f>-\infty\).
	\begin{itemize}
	\item ``\ref{thm:FY:ineq}''
		If \(x\notin\dom f\) the claim is obvious.
		Otherwise, the inequality follows from the definition of \(\conj{f}(y)\), cf. \eqref{eq:conj}.

	\item ``\ref{thm:FY:biconj<=}''
		For any \(\bar x\in X\) we have
		\[
			\biconj f(\bar x)
		\defeq
			\sup_{y\in Y}\set{\Phikernel(\bar x,y)-\conj{f}(y)}
		\leq
			\sup_{y\in Y}\set{f(\bar x)}
		=
			f(\bar x),
		\]
		where the inequality follows from assertion \ref{thm:FY:ineq}.

	\item ``\ref{thm:FY:triconj}''
		That \(\conj{f}\geq\conj{\biconj f}\) follows from assertion \ref{thm:FY:biconj<=} (with \(\Phikernel*\) replaced by \(\Phikernel\)).
		Conversely, for any \(\bar y\in Y\) we have
		\[
			\conj{\biconj f}(\bar y)
		\defeq
			\sup_{x\in X}\set{\Phikernel(x,\bar y)-\biconj{f}(x)}
		\geq
			\sup_{x\in X}\set{\Phikernel(x,\bar y)-f(x)}
		\defeq
			\conj f(\bar y),
		\]
		where the inequality again follows from assertion \ref{thm:FY:biconj<=}.
	\qedhere
	\end{itemize}\let\qed\relax
\end{proof}

The following result represents the most powerful tool of \(\Phikernel\)-conjugacy, as it demonstrates that the standard Fenchel duality result of convex analysis \cite[Thm. 23.5]{rockafellar1970convex} remains valid for arbitrary coupling functions \(\Phikernel\).
Compared to the formulation of \cite[Lem. 2.6]{laude2023dualities}, a fourth statement in the equivalence is added and the lemma is otherwise identical.

\begin{lemma}[Fenchel \(\Phikernel\)-duality, {\cite[Lem. 2.6]{laude2023dualities}}]\label{thm:FY}%
	For any proper function \(\func{f}{X}{\Rinf}\) and \((\bar x,\bar y)\in X\times Y\), the following are equivalent:
	\begin{enumerateq}
	\item \label{thm:FY:subdiff}%
		\(\bar y\in\subdiff f(\bar x)\);
	\item \label{thm:FY:subdiff*}%
		\(\bar x\in\subdiff* \conj{f}(\bar y)\cap\dom\subdiff f\);
	\item \label{thm:FY:=}%
		\(f(\bar x)+\conj f(\bar y)=\Phikernel(\bar x,\bar y)\);
	\item \label{thm:FY:biconj}%
		\(f(\bar x)=\biconj{f}(\bar x)\) and \(\bar y\in\subdiff \biconj{f}(\bar x)\).
	\end{enumerateq}
\end{lemma}
\begin{proof}~
	\begin{itemize}
	\item ``\ref{thm:FY:subdiff} \(\Leftrightarrow\) \ref{thm:FY:=}''
		This follows by definition, cf. \eqref{eq:conj} and \eqref{eq:subdiff}.

	\item ``(\ref{thm:FY:subdiff} \(\wedge\) \ref{thm:FY:=}) \(\Rightarrow\) \ref{thm:FY:subdiff*}''
		That \(\bar x\in\dom\subdiff f\) is obvious.
		The identity in the statement of assertion \ref{thm:FY:=} combined with \cref{thm:FY:biconj<=} implies that
		\(
			\biconj f(\bar x)+\conj f(\bar y)\leq\Phikernel(\bar x,\bar y)
		\).
		By \cref{thm:FY:ineq}, equality must hold, hence
		\[
			f(\bar x)+\conj f(\bar y)
		=
			\Phikernel(\bar x,\bar y)
		=
			\biconj f(\bar x)+\conj f(\bar y).
		\]
		By virtue of the shown equivalence ``\ref{thm:FY:subdiff} \(\Leftrightarrow\) \ref{thm:FY:=}'' (with \((\Phikernel,f)\) replaced by \((\Phikernel*,\conj f)\)), the inclusion \(\bar x\in\subdiff*\conj f(\bar y)\) follows.

	\item ``\ref{thm:FY:subdiff*} \(\Rightarrow\) \ref{thm:FY:biconj}''
		The inclusion \(\bar x\in\dom\subdiff f\) entails the existence of \(\bar\eta\in\subdiff f(\bar x)\), which by the equivalence ``\ref{thm:FY:subdiff} \(\Leftrightarrow\) \ref{thm:FY:=}'' implies that
		\(
			f(\bar x)+\conj f(\bar\eta)
		=
			\Phikernel(\bar x,\bar\eta)
		\).
		The same equivalence, this time with \((\Phikernel,f)\) replaced by \((\Phikernel*,\conj f)\), combined with the inclusion \(\bar x\in\subdiff*\conj f(\bar y)\) implies that
		\begin{align*}
			\Phikernel(\bar x,\bar y)
			-
			\conj f(\bar y)
		=
			\biconj f(\bar x)
		={} &
			\sup_{y\in Y}\set{
				\Phikernel(\bar x,y)
				-
				\conj f(y)
			}
		\\
		\geq{} &
			\Phikernel(\bar x,\bar\eta)
			-
			\conj f(\bar\eta)
		=
			f(\bar x)
		\geq
			\biconj f(\bar x),
		\end{align*}
		where the last inequality follows from \cref{thm:FY:biconj<=}.
		Hence, equality holds throughout, and in particular \(f(\bar x)=\biconj f(\bar x)\).
		By \cref{thm:FY:triconj}, one also has that
		\(
			\biconj f(\bar x)
			+
			\conj{\biconj f}(\bar y)
		=
			\Phikernel(\bar x,\bar y)
		\),
		and as before we conclude that \(\bar y\in\subdiff\biconj f(\bar x)\).

	\item ``\ref{thm:FY:biconj} \(\Rightarrow\) \ref{thm:FY:=}''
		In light of the equivalence ``\ref{thm:FY:subdiff} \(\Leftrightarrow\) \ref{thm:FY:=}'' (with \(f\) replaced by \(\biconj f\)), the inclusion \(\bar y\in\subdiff\biconj f(\bar x)\) implies that
		\[
			\Phikernel(\bar x,\bar\eta)
		=
			\biconj f(\bar x)+\conj{\biconj f}(\bar\eta)
		=
			\biconj f(\bar x)+\conj f(\bar\eta)
		=
			f(\bar x)+\conj f(\bar\eta)
		\]
		as claimed, where the second identity follows from \cref{thm:FY:triconj}.
	\qedhere
	\end{itemize}\let\qed\relax
\end{proof}

		\subsection{Bregman-Moreau envelopes as \texorpdfstring{\(\Phikernel\)}{Phi}-conjugates}\label{sec:Phiconj:BM}%
			By setting \(\Phikernel=-\tfrac{1}{\lambda}\D\) one recovers familiar objects treated in this manuscript; the verification of the following identities is straightforward from the respective definitions.

\begin{corollary}\label{thm:subdiff}%
	Fix \(\lambda>0\) and set \(\Phikernel=-\frac{1}{\lambda}\D\).
	Then, for any \(\func{f}{\X}{\Rinf}\) the following hold:
	\begin{enumerate}
	\item \label{thm:subdiff:env}%
		\(\conj f=-\env{f}\).
	\item \label{thm:subdiff:hull}%
		\(\biconj f=\hull_{\lambda}{f}\).
	\item \label{thm:subdiff:prox}%
		\(\prox_{\lambda f}=(\subdiff f)^{-1}\); that is,
		\(\bar y\in\subdiff f(\bar x)\)
		if and only if
		\(\bar x\in\prox_{\lambda f}(\bar y)\).
	\end{enumerate}
	Similarly, for any \(\func{g}{\Y}{\Rinf}\) the following hold:
	\begin{enumerate}[resume]
	\item
		\(\conj*g=-\env*{g}\).
	\item
		\(\biconj*g=\hull*_{\lambda}{g}\).
	\item \label{thm:subdiff*:prox*}%
		\(\prox*_{\lambda g}=(\subdiff*g)^{-1}\); that is,
		\(\bar x\in\subdiff*g(\bar y)\)
		if and only if
		\(\bar y\in\prox*_{\lambda g}(\bar x)\).
	\end{enumerate}
\end{corollary}

In light of the above relations, \cref{thm:FY:biconj<=,thm:FY:triconj} specialized to \(\Phikernel=-\frac{1}{\lambda}\D\) confirm that, for any proper functions \(\func{f}{\X}{\Rinf}\) and \(\func{g}{\Y}{\Rinf}\), one has that
\begin{equation}\label{eq:Phi:hull}
	\hull{f}
\leq
	f
\quad\text{and}\quad
	\hull*{g}
\leq
	g,
\end{equation}
as well as
\begin{equation}\label{eq:Phi:trienv}
	\env{f}
=
	\env{(\hull{f})}
\quad\text{and}\quad
	\env*{g}
=
	\env*{(\hull*{g})}.
\end{equation}
Similarly, \cref{thm:FY} reads as follows:

\begin{corollary}\label{thm:FY:BM}%
	Let \(\func{f}{\X}{\Rinf}\).
	For any \(\lambda>0\), \(\bar x\in\X\), and \(\bar y\in\Y\), the following three statements are equivalent:
	\begin{enumerateq}
	\item
		\(\bar x\in\prox_{\lambda f}(\bar y)\);
	\item
		\(\bar y\in\prox*_{-\lambda\env{f}}(\bar x)\) and \(\bar x\in\range\prox_{\lambda f}\);
	\item
		\(f(\bar x)=\hull_{\lambda}{f}(\bar x)\) and \(\bar x\in\prox_{\lambda\hull_{\lambda}{f}}(\bar y)\).
	\end{enumerateq}
	Similarly, for \(\func{g}{\Y}{\Rinf}\) the following are equivalent:
	\begin{enumerateq}[label=\textup{(\alph*')}]
	\item
		\(\bar y\in\prox*_{\lambda g}(\bar x)\);
	\item
		\(\bar x\in\prox_{-\lambda\env*{g}}(\bar y)\) and \(\bar y\in\range\prox*_{\lambda g}\);
	\item
		\(g(\bar y)=\hull*_{\lambda}{g}(\bar y)\) and \(\bar y\in\prox*_{\lambda\hull*_{\lambda}{g}}(\bar x)\).
	\end{enumerateq}
\end{corollary}

Finally, note that (left) \(\kernel\)-prox-boundedness of \(\func{f}{\X}{\Rinf}\) as in \cref{def:PB} corresponds to the existence of a ``(left) \(\Phikernel\)-support'' with \(\Phikernel=-\tfrac{1}{\lambda}\D\) for some \(\lambda>0\); namely, a pair \((y,\beta)\in\Y\times\R\) such that
\(
	f\geq\Phikernel({}\cdot{},y)-\beta
\).
Similarly, right \(\kernel\)-prox-boundedness of \(\func{g}{\Y}{\Rinf}\) as in \cref{def:PB*} corresponds to the existence of a ``(right) \(\Phikernel\)-support''  with \(\Phikernel=-\tfrac{1}{\lambda}\D\) for some \(\lambda>0\), hence a pair \((x,\beta)\in\X\times\R\) such that
\(
	g
\geq
	\Phikernel(x,{}\cdot{})-\beta
\).

Analogously to lsc convex functions for which minorants are affine, functions which can be expressed as supremum of left/right \(\Phikernel\)-minorants are called \emph{left/right \(\Phikernel\)-convex}.
It is not difficult to verify that these functions are precisely negative envelopes, at which point the relation \eqref{eq:Phi:trienv} confirms the familiar claim that, for an lsc function, being (left/right \(\Phikernel\)-)convex is equivalent to agreeing with the (left/right \(\Phikernel\)-)biconjugate.
We do not pursue these remarks any further here, as they are not central to the discussion that follows; for a thorough treatment we refer the interested reader to \cite{laude2023dualities}.

		\subsection{A new characterization of relative smoothness}
			A fundamental result in convex analysis states that for any \(\func{f}{\R^n}{\R}\) convex and differentiable, the following statements are equivalent:
\begin{enumerateq}
\item
	\(f\) is \emph{1-smooth}: namely, \(\nabla f\) is Lipschitz continuous with modulus 1;
\item
	\(\D_f(x,y)\leq\tfrac{1}{2}\norm{x-y}^2\)
	for any \(x,y\in\R^n\).
\item
	\(\j-f\) is convex;
\item \label{item:coco}%
	\(
		\D_f(x,y)
	\geq
		\tfrac{1}{2}
		\norm{\nabla f(x) - \nabla f(y)}^2
	\)
	for any \(x,y\in\R^n\);
\end{enumerateq}
see \cite[Thm. 2.1.5]{nesterov2018lectures}.
The second characterization is the celebrated \emph{quadratic upper bound} that has been ubiquitously exploited for establishing descent of gradient-based minimization algorithms.
The third statement has found important applications in the last decade, as it led to generalizations in the Bregman setting \cite{bauschke2017descent,bolte2018first}, obtained by replacing \(\j\) with a more general dgf \(\kernel\).

Thus, a proper and lsc function \(\func{f}{\X}{\Rinf}\) is \emph{smooth relative to \(\kernel\)}, or \emph{\B-smooth} for short, if \(f\) is differentiable on \(\interior\X\) and \(\kernel-f\) is convex on \(\interior\X\).
The terminology ``\B-smooth'' adopted here is borrowed from \cite{laude2023dualities}, which investigates the duality gap between relative smoothness and \emph{relative strong convexity} (\emph{\B-strong convexity}) for \(\kernel\) other than squared Euclidean norms.
The ``B'' in \B{}, short for ``Bregman'', distinguishes these notions from the corresponding \emph{anisotropic} counterparts that are shown to complete the duality picture.

The chain of equivalences thus far for a convex function \(f\) becomes
\[
	f \text{ \B-smooth}
~~\Leftrightarrow~~
	\kernel-f \text{ convex}
~~\Leftrightarrow~~
	\D_f(x,y)\leq\D(x,y)
	~
	\forall x,y\in\R^n,
\]
in which the first one is a mere matter of definition, and the second one is self-apparent from the fact that \(\D-\D_f=\D_{\kernel-f}\).
The question whether an equivalent of \ref{item:coco} could be included in the picture has only partially been addressed in \cite[Lem. 3]{dragomir2021fast}, see also \cite[Prop. 5.5]{dragomir2021methodes}, via the introduction of the following Bregman variant.

\begin{definition}[\B-cocoercivity inequality]\label{def:Bcoco}%
	Suppose that \(\kernel\) is Legendre and 1-coercive, and let \(\func{f}{\X}{\Rinf}\) be proper and lsc.
	We say that \(f\) \emph{satisfies the \B-cocoercivity inequality} if it is differentiable on \(\interior\X\) and
	\begin{equation}\label{eq:Bcoco}
		\D_{f}(\bar x,x)
	\geq
		\D*\bigl(
			\nabla\kernel(\bar x)-(\nabla f(\bar x)-\nabla f(x))
			,\,
			\nabla\kernel(\bar x)
		\bigr)
	\quad
		\forall x,\bar x\in\interior\X.
	\end{equation}
\end{definition}

It is shown in \cite[Prop. 5.5]{dragomir2021methodes} that this property is necessary for a convex \(f\) to be \B-smooth, but whether it is also sufficient remains an open question.
When \(\kernel=\j\), this inequality effectively reduces to \ref{item:coco}, and its well-known sufficiency for 1-smoothness is easy to verify.
Indeed, summing the two inequalities obtained by interchanging \(x\) and \(\bar x\) in \eqref{eq:Bcoco} leads to the celebrated Baillon-Haddad theorem \cite{baillon1977quelques}, namely the fact that
\[
	\innprod{\nabla f(x)-\nabla f(\bar x)}{x-\bar x}
\geq
	\norm{\nabla f(x)-\nabla f(\bar x)}^2
\quad
	\forall x,\bar x\in\R^n,
\]
and 1-smoothness follows from an easy application of the Cauchy-Schwarz inequality.
Unless \(\kernel\) is quadratic, however, the same arguments are not applicable.

In the following result we partially fill this gap.
We conjecture that the \B-coercivity inequality as in \cref{def:Bcoco} is also sufficient for \B-smoothness, and prove that this is indeed true at least whenever \(\dom\lsubdiff f\subseteq\interior\X\), as is the case when \(\X\) is open.
In this case, indeed, we show that the \B-cocoercivity inequality for \(f\) is precisely the \emph{anisotropic strong convexity} inequality for its conjugate relative to \(\kernel*\) \cite[Def. 3.8]{laude2023dualities}, \a*-strong convexity for short, after simple algebraic manipulations.
More generally, regardless of the validity of the condition \(\dom\lsubdiff f\subseteq\interior\X\), we identify \B-smoothness with the following \emph{extended} \B-cocoercivity inequality which also involves subgradients at boundary points.

\begin{definition}[extended \B-cocoercivity]\label{def:Bcoco+}%
	Suppose that \(\kernel\) is Legendre and 1-coercive, and let \(\func{f}{\X}{\Rinf}\) be proper, lsc, and differentiable on \(\interior\X\).
	We say that \(f\) \emph{satisfies the extended \B-cocoercivity inequality} if
	\begin{multline}\label{eq:Bcoco+}
		f(\bar x)-f(x)-\innprod{\xi}{\bar x-x}
	\geq
		\D*\bigl(
			\nabla\kernel(\bar x)-(\nabla f(\bar x)-\xi)
			,\,
			\nabla\kernel(\bar x)
		\bigr)
	\\
		\forall\bar x\in\interior\X,\ (x,\xi)\in\graph\lsubdiff f.
	\end{multline}
\end{definition}

The implications ``\ref{thm:Bcoco:Bsmooth} \(\Leftrightarrow\) \ref{thm:Bcoco:Phiconj} \(\Leftrightarrow\) \ref{thm:Bcoco:astrcvx}'' in the following characterization have already been shown in \cite[Prop. 5.11]{laude2025anisotropic}, through the adoption of \emph{extended arithmetics} to resolve indeterminate forms \(\infty-\infty\).
These arise as a consequence of the standard convention of considering functions on the entire space \(\R^n\), regardless of what the domain of \(\kernel\) is.
As a testimony of the simplifications incurred by \emph{our} convention, as well as to make this section self-contained, we replicate the proof with ``standard'' arithmetics.

\begin{theorem}[\B-smoothness and cocoercivity; \normalfont extension of {\cite[Prop. 5.11]{laude2025anisotropic}}]\label{thm:Bcoco}%
	\RenewDocumentCommand{\Phikernel}{s}{\IfBooleanTF{#1}{\hat\oldPsi}{\hat\oldPhi}}%
	Let \(\kernel\) be Legendre and 1-cocoercive, and let \(\func{\tilde f}{\R^n}{\Rinf}\) be the canonical extension (see \cref{def:fext}) of a proper, lsc, convex function \(\func{f}{\X}{\Rinf}\) such that \(\dom f\cap\interior\X\neq\emptyset\).
	Then, the following are equivalent:
	\begin{enumerateq}
	\item \label{thm:Bcoco:Bsmooth}%
		\(f\) is \B-smooth.
	\item \label{thm:Bcoco:Phiconj}%
		\(\dom f=\X\) and \(f=\kernel-\conj{\tilde f^*}{}^*(-{}\cdot{})\) on \(\interior\X\), where \(\Phikernel\coloneqq\kernel*({}\cdot-\cdot{})\).
	\item \label{thm:Bcoco:Bcoco+}%
		\(f\) satisfies the extended \B-cocoercivity inequality \eqref{eq:Bcoco+}.
	\item \label{thm:Bcoco:astrcvx}%
		The Fenchel conjugate \(\tilde f^*\) satisfies the \a*-strong convexity inequality
		\begin{multline}\label{eq:astrcvx}
			\tilde f^*(\xi)
		\geq
			\tilde f^*(\bar\xi)
			+
			\kernel*\bigl(\xi-\bar\xi+\nabla\kernel(\bar x)\bigr)
			-
			\kernel*(\nabla\kernel(\bar x))
		\\
			\forall\xi,\bar\xi\in\R^n,\bar x\in\lsubdiff\tilde f^*(\bar\xi)\cap\interior\X.
		\end{multline}
	\end{enumerateq}
	All these equivalent conditions imply the following:
	\begin{enumerateq}[resume]
	\item \label{thm:Bcoco:Bcoco}%
		\(f\) (is differentiable on \(\interior\X\) and) satisfies the \B-cocoercivity inequality \eqref{eq:Bcoco}.%
	\end{enumerateq}
	A sufficient condition for the converse implication to hold is having \(\dom\lsubdiff f\subseteq\interior\X\) (equivalently, \(\dom\lsubdiff f=\dom\nabla f=\interior\X\)), as is the case when \(\X\) is open or when \(f\) is essentially smooth.
\end{theorem}
\begin{proof}
	We begin by commenting that the implication ``\ref{thm:Bcoco:Bcoco+} \(\Rightarrow\) \ref{thm:Bcoco:Bcoco}'' is obvious, since \eqref{eq:Bcoco} amounts to \eqref{eq:Bcoco+} with both \(x\) and \(\bar x\) restricted to \(\interior\X\).
	For this reason, whenever \(\dom\lsubdiff f=\dom\nabla f=\interior\X\) it is clear that \eqref{eq:Bcoco} and \eqref{eq:Bcoco+}, and in particular assertions \ref{thm:Bcoco:Bcoco+} and \ref{thm:Bcoco:Bcoco}, are equivalent.
	We now prove the equivalences among the first four statements without this assumption.
	\begin{itemize}
	\item ``\ref{thm:Bcoco:Bsmooth} \(\Rightarrow\) \ref{thm:Bcoco:Bcoco+}''
		This is a minor modification of \cite[Prop. 5.5]{dragomir2021methodes}; for self-containedness we provide a simplified proof.
		Fix \(\bar x\in\interior\X\) and \((x,\xi)\in\graph\lsubdiff f\), and denote \(\D_{f,\xi}({}\cdot{},x)\coloneqq f-f(x)-\innprod{\xi}{{}\cdot{}-x}\).
		By direct computation, it is easy to obtain the following generalization of the three-point identity \cite[Lem. 3.1]{chen1993convergence}, holding for any \(u\in\X\):
		\[
			\D_{f,\xi}(u,x)
		=
			\D_f(u,\bar x)
			+
			\D_{f,\xi}(\bar x,x)
			+
			\innprod{u-\bar x}{\nabla f(\bar x)-\xi}.
		\]
		Note that \(\hat u\coloneqq\nabla\kernel*\bigl(\nabla\kernel(\bar x)-(\nabla f(\bar x)-\xi)\bigr)\) is such that
		\(
			\nabla f(\bar x)-\xi
		=
			-\bigl(\nabla\kernel(\hat u)-\nabla\kernel(\bar x)\bigr)
		\),
		so that for \(u=\hat u\) the above identity simplifies as
		\begin{align*}
			\underbracket[0.5pt]{\D_{f,\xi}(\hat u,x)}_{\geq0}
		={} &
			\D_f(\hat u,\bar x)
			+
			\D_{f,\xi}(\bar x,x)
			-
			\DD(\hat u, \bar x)
		\\
		={} &
			-
			\smash{\overbracket[0.5pt]{\D_{\kernel-f}(\hat u,\bar x)}^{\geq0}}
			+
			\D_{f,\xi}(\bar x,x)
			-
			\D*(\nabla\kernel(\hat u),\nabla\kernel(\bar x)),
		\end{align*}
		where with \(\DD(\hat u,\bar x)=\D(\hat u,\bar x)+\D(\bar x,\hat u)\) we denote the symmetrized Bregman distance between \(\hat u\) and \(\bar x\).
		This yields the desired extended \B-cocoercivity inequality \eqref{eq:Bcoco+}.

	\item ``\ref{thm:Bcoco:Bcoco+} \(\Rightarrow\) \ref{thm:Bcoco:astrcvx}''
		\Cref{thm:cvxext} yields that \(f=\tilde f^{**}\restr_{\X}\), hence that
		\[
			\graph\lsubdiff f
		=
			\graph\lsubdiff\tilde f^{**}
			\cap
			(\X\times\R^n)
		=
			\graph(\lsubdiff\tilde f^*)^{-1}
			\cap
			(\X\times\R^n).
		\]
		It follows from the assumptions that the condition \(\bar x\in\lsubdiff\tilde f^*(\bar\xi)\cap\interior\X\) is equivalent to \(\bar x\in\interior\X\) and \(\bar\xi=\nabla f(\bar x)\), hence that \eqref{eq:astrcvx} reads
		\begin{multline}\label{eq:astrcvx:nablaf}
			\tilde f^*(\xi)
		\geq
			\tilde f^*(\nabla f(\bar x))
			+
			\kernel*\bigl(\xi-\nabla f(\bar x)+\nabla\kernel(\bar x)\bigr)
			-
			\kernel*(\nabla\kernel(\bar x))
		\\
			\forall\xi\in\R^n,\bar x\in\interior\X
		\end{multline}
		in this case.
		In what follows, let \(\bar x\in\interior\X\) be arbitrary.
		Note that \eqref{eq:astrcvx:nablaf} is vacuously true for \(\xi\notin\dom\tilde f^*\).
		Next, suppose that \(\xi\in\dom\lsubdiff\tilde f^*\); the boundary case will be treated last.
		Observe that
		\begin{align*}
		&
			\D*\bigl(
				\nabla\kernel(\bar x)-(\nabla f(\bar x)-\xi)
				,\,
				\nabla\kernel(\bar x)
			\bigr)
		\\
		={} &
			\kernel*\bigl(\nabla\kernel(\bar x)-(\nabla f(\bar x)-\xi)\bigr)
			-
			\kernel*(\nabla\kernel(\bar x))
			+
			\innprod{\bar x}{\nabla f(\bar x)-\xi},
		\end{align*}
		whereas, for any \(x\in\lsubdiff\tilde f^*(\xi)\),
		\begin{align*}
			f(\bar x)-f(x)-\innprod{\xi}{\bar x-x}
		={} &
			f(\bar x)+\tilde f^*(\xi)-\innprod{\xi}{\bar x}
		\\
		={} &
			\tilde f^*(\xi)
			-
			\tilde f^*(\nabla f(\bar x))
			+
			\innprod{\bar x}{\nabla f(\bar x)-\xi}.
		\end{align*}
		By using these identities in the restricted \B-cocoercivity inequality \eqref{eq:Bcoco+}, the claimed \eqref{eq:astrcvx:nablaf} is obtained.

		Finally, points \(\xi\in\dom\tilde f^*\setminus\dom\lsubdiff\tilde f^*\) belong to the relative boundary of \(\tilde f^*\).
		Since \(\kernel*\) is also continuous (on \(\R^n\)) and since \(\tilde f^*\) is proper, convex, and lsc, a simple limiting argument along a line segment stemming from any \(\bar\eta\in\relint\dom\tilde f^*\) as in \cite[Thm. 7.5]{rockafellar1970convex} yields the needed inequality holding for any \(\xi\in\R^n\).

	\item ``\ref{thm:Bcoco:astrcvx} \(\Rightarrow\) \ref{thm:Bcoco:Phiconj}''
		\RenewDocumentCommand{\Phikernel}{s}{\IfBooleanTF{#1}{\hat\oldPsi}{\hat\oldPhi}}%
		The assumptions ensure that \(\kernel*\) is Legendre and with full domain, and that \(\emptyset\neq\relint\dom f\subseteq\dom\lsubdiff f\cap\interior\X\).
		As a consequence of convexity and lower semicontinuity of \(f\), it follows from \cref{thm:cvxext} that
		\[
			\tilde f^{**}\restr_{\X}=f,
		\]
		hence that
		\[
			\graph\lsubdiff f
		=
			\graph\lsubdiff\tilde f^{**}
			\cap
			(\X\times\R^n)
		=
			\graph(\lsubdiff\tilde f^*)^{-1}
			\cap
			(\X\times\R^n).
		\]
		Let \(\bar x\in\dom\lsubdiff f\cap\interior\X\) be fixed, and pick any \(\bar\xi\in\lsubdiff f(\bar x)=\lsubdiff\tilde f(\bar x)\) (in particular, \(\bar x\in\lsubdiff\tilde f^*(\bar\xi)\)).
		Denote \(\bar\eta\coloneqq\bar\xi-\nabla\kernel(\bar x)\), so that the \a*-strong convexity inequality \eqref{eq:astrcvx} reads
		\begin{equation}\label{eq:astrcvx:bar}
			\tilde f^*(\xi)
		\geq
			\tilde f^*(\bar\xi)
			+
			\kernel*\bigl(\xi-\bar\eta\bigr)
			-
			\kernel*(\bar\xi-\bar\eta)
		\quad
			\forall\xi\in\R^n.
		\end{equation}
		In terms of the couplings \(\func{\Phikernel,\Phikernel*}{\R^n}{\R}\) defined as
		\[
			\Phikernel(\xi,\eta)
		\coloneqq
			\kernel*(\xi-\eta)
		\eqqcolon
			\Phikernel*(\eta,\xi)
		\quad
			\forall\xi,\eta\in\R^n,
		\]
		according to \cref{def:subdiff} the inequality \eqref{eq:astrcvx:bar} means that \(\bar\eta\in\subdiff\tilde f^*(\bar\xi)\).
		It then follows from \cref{thm:FY} that
		\[
			\tilde f^*(\bar\xi)
		=
			\biconj{(\tilde f^*)}(\bar\xi)
		=
			\sup_{\eta\in\R^n}\set{\kernel*(\bar\xi-\eta)-\conj{\tilde f^*}(\eta)}
		=
			\sup_{\eta\in\R^n}\sup_{x\in\R^n}H(x,\eta),
		\]
		where
		\[
			H(x,\eta)
		\coloneqq
			\innprod{\bar\xi-\eta}{x}-\kernel(x)-\conj{\tilde f^*}(\eta).
		\]
		Note that
		\[
			\nabla\kernel*(\bar\xi-\eta)
		\in
			\argmax H({}\cdot{},\eta).
		\]
		Moreover, the inclusion \(\bar\eta\in\subdiff\tilde f^*(\bar\xi)\) also implies through \cref{thm:FY} that \(\bar\xi\in\subdiff*\conj{\tilde f^*}(\bar\eta)\), hence that
		\[
			\bar\eta
		\in
			\argmax_\eta
			\set{\kernel*(\bar\xi-\eta)-\conj{\tilde f^*}(\eta)}
		=
			\argmax_\eta
			H(\nabla\kernel*(\bar\xi-\eta),\eta).
		\]
		The two inclusions combined yield that
		\(
			(\bar x,\bar\eta)
		=
			(\nabla\kernel*(\bar\xi-\bar\eta),\bar\eta)
		\in
			\argmax H
		\),
		resulting in
		\begin{align*}
			\tilde f^*(\bar\xi)
		=
			H(\bar x,\bar\eta)
		={} &
			\sup_{\eta\in\R^n}H(\bar x,\eta)
		\\
		={} &
			\sup_{\eta\in\R^n}\set{\innprod{\bar\xi-\eta}{\bar x}-\kernel(\bar x)-\conj{\tilde f^*}(\eta)}
		\\
		={} &
			\innprod{\bar\xi}{\bar x}-\kernel(\bar x)+\conj{\tilde f^*}{}^*(-\bar x).
		\end{align*}
		Rearranging (which is legit since \(\innprod{\bar\xi}{\bar x}-\kernel(\bar x)\in\R\)),
		\[
			\conj{\tilde f^*}{}^*(-\bar x)
		=
			\kernel(\bar x)-\innprod{\bar\xi}{\bar x}+\tilde f^*(\bar\xi)
		=
			\kernel(\bar x)-\tilde f^{**}(\bar x)
		=
			\kernel(\bar x)-f(\bar x),
		\]
		where the last two identities use the fact that \((\bar\xi,\bar x)\in\graph\lsubdiff\tilde f^*\) and that \(\bar x\in\X\).
		To conclude, observe that \eqref{eq:astrcvx:bar} (for a fixed \(\bar x\in\interior\X\cap\dom\lsubdiff f\) and \(\bar\xi\in\lsubdiff f(\bar x)\), and with \(\bar\eta\coloneqq\bar\xi-\nabla\kernel(\bar x)\)) implies that, for any \(x\in\X\),
		\begin{align*}
			f(x)
		=
			\tilde f^{**}(x)
		={} &
			\sup_{\xi\in\R^n}\set{
				\innprod{x}{\xi}-\tilde f^*(\xi)
			}
		\\
		\leq{} &
			\sup_{\xi\in\R^n}\set{
				\innprod{x}{\xi}
				-
				\kernel*\bigl(\xi-\bar\eta\bigr)
			}
			+
			\kernel*(\bar\xi-\bar\eta)
			-
			\tilde f^*(\bar\xi)
		\\
		={} &
			\kernel(x)
			+
			\innprod{\bar\eta}{x}
			+
			\underbracket[0.5pt]{
				\kernel*(\bar\xi-\bar\eta)
				-
				\tilde f^*(\bar\xi)
			}_{\text{finite constant}},
		\end{align*}
		and in particular that \(\dom f=\dom\tilde f=\dom\kernel=\X\).

	\item ``\ref{thm:Bcoco:Phiconj} \(\Rightarrow\) \ref{thm:Bcoco:Bsmooth}''
		Obvious, since \(\kernel-f\) equals the lsc and convex function \(\conj{\tilde f^*}{}^*(-{}\cdot{})\) on \(\interior\X\), which is also proper since \(f\) and \(\kernel\) are finite on \(\X\).
	\qedhere
	\end{itemize}\let\qed\relax
\end{proof}%

			Surprisingly, when \(\kernel\) has not full domain, or equivalently when \(\kernel*\) is not 1-coercive, an \a*-strongly convex function may fail to be essentially strictly convex.
In light of the equivalence shown in \cref{thm:Bcoco}, another way to phrase this is that, unless the Legendre and 1-coercive dgf \(\kernel\) has full domain, the conjugate of a convex function that is smooth relative to \(\kernel\) is not necessarily essentially strictly convex:
\begin{equation}
	f \text{ \B-smooth}
~\not\Rightarrow~
	\tilde f^* \text{ essentially strictly convex.}
\end{equation}
This sharpens the duality gap observed in \cite{laude2023dualities} for the full-domain setting, where although \B*-strong convexity of \(f^*\) is shown to be unnecessary for \B-smooth\-ness of \(f\), essential strict convexity is still guaranteed; see \cite[Prop. 4.1]{laude2023dualities}.
The following counterexample showcases this sharper gap.

\begin{example}[\a*-strong convexity without essential strict convexity]%
	Consider \(\func{\kernel}{\R^2}{\Rinf}\) given by \(\kernel(x)=\tfrac{x_1^2}{4x_2}+\tfrac{1}{2}x_2^2-\ln x_2\) on \(\X=\R\times\R_{++}\), and \(\infty\) elsewhere.
	It is easy to verify that \(\kernel\) is convex, Legendre and 1-coercive, and that
	\begin{equation}\label{eq:Y}
		\kernel*(\xi)
	=
		\tfrac{1}{2}Y(\xi)^2-1+\ln Y(\xi)
	\quad\text{with}\quad
		Y(\xi)
	\coloneqq
		\frac{
			\xi_1^2+\xi_2+\sqrt{(\xi_1^2+\xi_2)^2+4}
		}{2}.
	\end{equation}
	The function \(\func{f}{\X}{\Rinf}\) given by \(f(x)=\tfrac{x_1^2}{4x_2}\) is convex and \B-smooth.
	As such, the conjugate of its canonical extension \(\func{\tilde f}{\R^2}{\Rinf}\) as in \cref{def:fext}, easily seen to be
	\[
		\tilde f^*
	=
		\indicator_D
	\quad\text{where}\quad
		D\coloneqq\set{\xi\in\R^2}[\xi_1^2+\xi_2\leq0],
	\]
	satisfies the \a*-strong convexity inequality \eqref{eq:astrcvx}, yet it is clearly not essentially strictly convex.
\end{example}

The culprit of this shortcoming is attributable to the restriction \(\bar x\in\interior\X\) in \eqref{eq:astrcvx}.
The function in the above example well showcases this fact.
Indeed, for any \(\bar\xi\in\interior D\), the only element of \(\lsubdiff\tilde f^*(\bar\xi)=\lsubdiff\indicator_D(\bar\xi)\) is \(\bar x=0\), which does not belong to \(\interior\X=\R\times\R_{++}\).
The only other possible choices for \(\bar\xi\) are elements of \(\boundary D=\set{(u,-u^2)}[u\in\R]\), at which
\[
	\lsubdiff\tilde f^*(\bar\xi)
=
	\lsubdiff\indicator_D(\bar\xi)
=
	\set{t(2u,1)}[t\geq0]
\quad
	\text{for }\bar\xi=(u,-u^2).
\]
We shall confirm the validity of \eqref{eq:astrcvx}.
For \(\bar x=t(2u,1)\in\interior\X\) with \(t>0\) one has that \(\nabla\kernel(\bar x)=\bigl(u,-u^2+t-\tfrac{1}{t}\bigr)\), hence
\[
	Y\bigl(\nabla\kernel(\bar x)\bigr)
=
	\frac{
		t-\tfrac{1}{t}+\sqrt{(t-\tfrac{1}{t})^2+4}
	}{2}
=
	t,
\]
so that
\(
	\kernel*\bigl(\nabla\kernel(\bar x)\bigr)
=
	\tfrac{1}{2}t^2-1+\ln t
\)
only depends on \(t>0\), but not on \(u\in\R\).
The \a*-strong convexity inequality \eqref{eq:astrcvx} then reads
\begin{equation}\label{eq:exastr}
	\kernel*(\nabla\kernel(\bar x))
=
	\tfrac{1}{2}t^2-1+\ln t
\geq
	\kernel*\bigl(\xi_1,\xi_2+t-\tfrac{1}{t}\bigr)
\quad
	\forall\xi\in D,t>0.
\end{equation}
Observing that
\begin{align*}
	Y\bigl(\xi_1,\xi_2+t-\tfrac{1}{t}\bigr)
={} &
	\frac{
		\xi_1^2+\xi_2+t-\tfrac{1}{t}+\sqrt{(\xi_1^2+\xi_2+t-\tfrac{1}{t})^2+4}
	}{2}
\\
\leq{} &
	\frac{
		t-\tfrac{1}{t}+\sqrt{(t-\tfrac{1}{t})^2+4}
	}{2}
=
	t
=
	Y\bigl(\nabla\kernel(\bar x)\bigr)
\quad
	\forall \xi\in D,
\end{align*}
and that \(\kernel*\) is increasing with respect to \(Y(\xi)\) in \eqref{eq:Y}, the validity of the \a*-strong convexity inequality \eqref{eq:exastr} is confirmed.

	\section{Conclusion}\label{sec:conclusion}
		In this work, we have revisited fundamental objects in Bregman-based optimization, the Bregman proximal mapping, the Bregman-Moreau envelope, and the Bregman proximal hull, through a shift in perspective that emphasizes their natural domains of definition.
Rather than adhering to the conventional viewpoint that functions must be defined on the full ambient space \(\R^n\), we have demonstrated that a more precise and flexible framework arises by instead considering functions defined only on appropriate subsets, such as \(\dom\kernel\) and \(\interior\dom\kernel\), determined by the distance-generating function \(\kernel\).

By avoiding unnecessary extensions of functions outside their meaningful domains, we were able to formulate more concise and natural versions of key results, highlight domain-sensitive issues that have led to oversights in prior literature, and, most importantly, considerably extend the applicability of known results under minimal assumptions.

All this is achieved by first meticulously revisiting known convex and variational analysis concepts for set-valued operators and extended-real-valued functions defined on \emph{subsets} of \(\R^n\), leading to new insights which we believe enjoy an independent appeal.

Several promising directions for future research emerge from the domain-aware perspective developed in this work.
One natural extension concerns the study of \emph{Klee envelopes} \cite{bauschke2009bregmanK}, where a refined treatment of domain restrictions may yield analogous benefits in terms of generality and conceptual clarity.
Another compelling direction involves relaxing the differentiability assumptions on the distance-generating function \(\kernel\).
Extending Bregman distances to accommodate nondifferentiable or merely convex generators, along the lines of the approach proposed in \cite{maddison2021dual}, could open the door to a more comprehensive treatment of Bregman objects.
We anticipate that both lines of inquiry would benefit from the rigorous domain-sensitive framework introduced in this paper, and could further reinforce its utility in the analysis of Bregman-type methods.

	\subsection*{Acknowledgements}
		We express our sincere gratitude to Prof. Xianfu Wang and Dr. Emanuel Laude for the many constructive comments and feedback, and for pointing us to many references which we had originally overlooked.
		We are also deeply grateful to Prof. Shizuo Kaji for many insightful discussions on topology which were instrumental in helping us rigorously frame the mathematical aspects of this work.
		Finally, we also thank the Associate Editor Prof. Regina Burachik for her comments and support, as well as the anonymous referees, whose meticulous reading, in-depth analysis, and detailed suggestions greatly enhanced the quality of this work.

	\phantomsection
	\addcontentsline{toc}{section}{References}
	\bibliographystyle{plain}
	\bibliography{TeX/references.bib}

\begin{thebibliography}{10}

\bibitem{ahookhosh2021bregman}
Masoud Ahookhosh, Andreas Themelis, and Panagiotis Patrinos.
\newblock A {B}regman forward-backward linesearch algorithm for nonconvex
  composite optimization: Superlinear convergence to nonisolated local minima.
\newblock {\em SIAM Journal on Optimization}, 31(1):653--685, 2021.

\bibitem{auslender2002asymptotic}
Alfred Auslender and Marc Teboulle.
\newblock {\em Asymptotic Cones and Functions in Optimization and Variational
  Inequalities}.
\newblock Springer Monographs in Mathematics. Springer New York, 2002.

\bibitem{baillon1977quelques}
Jean-Bernard Baillon and Georges Haddad.
\newblock Quelques propri{\'e}t{\'e}s des op{\'e}rateurs angle-born{\'e}s et
  \(n\)-cycliquement monotones.
\newblock {\em Israel Journal of Mathematics}, 26:137--150, 1977.

\bibitem{bauschke2017descent}
Heinz~H. Bauschke, Jérôme Bolte, and Marc Teboulle.
\newblock A descent lemma beyond {L}ipschitz gradient continuity: First-order
  methods revisited and applications.
\newblock {\em Mathematics of Operations Research}, 42(2):330--348, 2017.

\bibitem{bauschke1997legendre}
Heinz~H. Bauschke and Jonathan~M. Borwein.
\newblock Legendre functions and the method of random {B}regman projections.
\newblock {\em Journal of Convex Analysis}, 4(1):27--67, 1997.

\bibitem{bauschke2003iterating}
Heinz~H. Bauschke and Patrick~L. Combettes.
\newblock Iterating {B}regman retractions.
\newblock {\em SIAM Journal on Optimization}, 13(4):1159--1173, 2003.

\bibitem{bauschke2017convex}
Heinz~H. Bauschke and Patrick~L. Combettes.
\newblock {\em Convex analysis and monotone operator theory in {H}ilbert
  spaces}.
\newblock CMS Books in Mathematics. Springer, 2017.

\bibitem{bauschke2018regularizing}
Heinz~H. Bauschke, Minh~N. Dao, and Scott~B. Lindstrom.
\newblock Regularizing with {B}regman--{M}oreau envelopes.
\newblock {\em SIAM Journal on Optimization}, 28(4):3208--3228, 2018.

\bibitem{bauschke2009bregmanC}
Heinz~H. Bauschke, Xianfu Wang, Jane Ye, and Xiaoming Yuan.
\newblock {B}regman distances and {C}hebyshev sets.
\newblock {\em Journal of Approximation Theory}, 159(1):3--25, 2009.

\bibitem{bauschke2009bregmanK}
Heinz~H. Bauschke, Xianfu Wang, Jane Ye, and Xiaoming Yuan.
\newblock {B}regman distances and {K}lee sets.
\newblock {\em Journal of Approximation Theory}, 158(2):170--183, 2009.
\newblock Special Issue in memory of Professor George G. Lorentz (1910–2006).

\bibitem{beck2003mirror}
Amir Beck and Marc Teboulle.
\newblock Mirror descent and nonlinear projected subgradient methods for convex
  optimization.
\newblock {\em Operations Research Letters}, 31(3):167--175, 2003.

\bibitem{benoist1996what}
Jo\"el Benoist and Jean-Baptiste Hiriart-Urruty.
\newblock What is the subdifferential of the closed convex hull of a function?
\newblock {\em SIAM Journal on Mathematical Analysis}, 27(6):1661--1679, 1996.

\bibitem{bolte2018first}
J{\'e}r{\^o}me Bolte, Shoham Sabach, Marc Teboulle, and Yakov Vaisbourd.
\newblock First order methods beyond convexity and {L}ipschitz gradient
  continuity with applications to quadratic inverse problems.
\newblock {\em SIAM Journal on Optimization}, 28(3):2131--2151, 2018.

\bibitem{burachik2021generalized1}
Regina~S. Burachik, Minh~N. Dao, and Scott~B. Lindstrom.
\newblock The generalized {B}regman distance.
\newblock {\em SIAM Journal on Optimization}, 31(1):404--424, 2021.

\bibitem{burachik2021generalized2}
Regina~S. Burachik, Minh~N. Dao, and Scott~B. Lindstrom.
\newblock Generalized {B}regman envelopes and proximity operators.
\newblock {\em Journal of Optimization Theory and Applications},
  190(3):744--778, 2021.

\bibitem{burachik2008setvalued}
Regina~S. Burachik and Alfredo~N. Iusem.
\newblock {\em Set-Valued Mappings and Enlargements of Monotone Operators}.
\newblock Springer, 2008.

\bibitem{cabot2017envelopes}
Alexandre Cabot, Abderrahim Jourani, and Lionel Thibault.
\newblock Envelopes for sets and functions: regularization and generalized
  conjugacy.
\newblock {\em Mathematika}, 63(2):383--432, 2017.

\bibitem{chen1993convergence}
Gong Chen and Marc Teboulle.
\newblock Convergence analysis of a proximal-like minimization algorithm using
  {B}regman functions.
\newblock {\em SIAM Journal on Optimization}, 3(3):538--543, 1993.

\bibitem{dragomir2021methodes}
Radu-Alexandru Dragomir.
\newblock {\em M{\'e}thodes de gradient de {B}regman pour probl{\`e}mes {\`a}
  r{\'e}gularit{\'e} relative}.
\newblock PhD thesis, Universit{\'e} de Toulouse, 2021.

\bibitem{dragomir2021fast}
Radu~Alexandru Dragomir, Mathieu Even, and Hadrien Hendrikx.
\newblock Fast stochastic {B}regman gradient methods: Sharp analysis and
  variance reduction.
\newblock In Marina Meila and Tong Zhang, editors, {\em Proceedings of the 38th
  International Conference on Machine Learning}, volume 139 of {\em Proceedings
  of Machine Learning Research}, pages 2815--2825. PMLR, 18- 2021.

\bibitem{figalli2023invitation}
Alessio Figalli and Federico Glaudo.
\newblock {\em An Invitation to Optimal Transport, Wasserstein Distances, and
  Gradient Flows}.
\newblock European Mathematical Society, 2023.

\bibitem{hiriarturruty1996convex}
Jean-Baptiste Hiriart-Urruty and Claude Lemar\'echal.
\newblock {\em Convex Analysis and Minimization Algorithms II: Advanced Theory
  and Bundle Methods}.
\newblock Grundlehren der mathematischen Wissenschaften. Springer Berlin
  Heidelberg, 1996.

\bibitem{kan2012moreau}
Chao Kan and Wen Song.
\newblock The {M}oreau envelope function and proximal mapping in the sense of
  the {B}regman distance.
\newblock {\em Nonlinear Analysis: Theory, Methods \& Applications},
  75(3):1385--1399, 2012.

\bibitem{latafat2022bregman}
Puya Latafat, Andreas Themelis, Masoud Ahookhosh, and Panagiotis Patrinos.
\newblock {B}regman {F}inito/{MISO} for nonconvex regularized finite sum
  minimization without {L}ipschitz gradient continuity.
\newblock {\em SIAM Journal on Optimization}, 32(3):2230--2262, 2022.

\bibitem{laude2021lower}
Emanuel Laude.
\newblock {\em Lower envelopes and lifting for structured nonconvex
  optimization}.
\newblock PhD thesis, Technische Universit{\"a}t M{\"u}nchen, 2021.

\bibitem{laude2020bregman}
Emanuel Laude, Peter Ochs, and Daniel Cremers.
\newblock {B}regman proximal mappings and {B}regman--{M}oreau envelopes under
  relative prox-regularity.
\newblock {\em Journal of Optimization Theory and Applications},
  184(3):724--761, 2020.

\bibitem{laude2025anisotropic}
Emanuel Laude and Panagiotis Patrinos.
\newblock Anisotropic proximal gradient.
\newblock {\em Mathematical Programming}, pages 1--45, 2025.

\bibitem{laude2023dualities}
Emanuel Laude, Andreas Themelis, and Panagiotis Patrinos.
\newblock Dualities for non-{E}uclidean smoothness and strong convexity under
  the light of generalized conjugacy.
\newblock {\em SIAM Journal on Optimization}, 33(4):2721--2749, 2023.

\bibitem{lu2018relatively}
Haihao Lu, Robert~M. Freund, and Yurii Nesterov.
\newblock Relatively smooth convex optimization by first-order methods, and
  applications.
\newblock {\em SIAM Journal on Optimization}, 28(1):333--354, 2018.

\bibitem{maddison2021dual}
Chris~J. Maddison, Daniel Paulin, Yee~Whye Teh, and Arnaud Doucet.
\newblock Dual space preconditioning for gradient descent.
\newblock {\em SIAM Journal on Optimization}, 31(1):991--1016, 2021.

\bibitem{martinezlegaz2005generalized}
Juan~Enrique Mart{\'i}nez-Legaz.
\newblock {\em Generalized Convex Duality and its Economic Applications}, pages
  237--292.
\newblock Springer New York, 2005.

\bibitem{nemirovskij1983problem}
Arkadij~S. Nemirovskij and David~B. Yudin.
\newblock {\em Problem complexity and method efficiency in optimization}.
\newblock Wiley-Interscience, 1983.

\bibitem{nesterov2018lectures}
Yurii Nesterov.
\newblock {\em Lectures on Convex Optimization}.
\newblock Springer, 2018.

\bibitem{ou2025linesearch}
Hongjia Ou, Puya Latafat, and Andreas Themelis.
\newblock Linesearch-free adaptive {B}regman proximal gradient for convex
  minimization without relative smoothness.
\newblock {\em arXiv:2508.01353}, 2025.

\bibitem{poliquin1996proxregular}
Ren\'e~A. Poliquin and R.~Tyrrell Rockafellar.
\newblock Prox-regular functions in variational analysis.
\newblock {\em Transactions of the American Mathematical Society},
  348(5):1805--1838, 1996.

\bibitem{rockafellar1970convex}
R.~Tyrrell Rockafellar.
\newblock {\em Convex {A}nalysis}.
\newblock Princeton university press, 1970.

\bibitem{rockafellar1998variational}
R.~Tyrrell Rockafellar and Roger~J.B. Wets.
\newblock {\em Variational Analysis}.
\newblock Springer, New York, 1998.

\bibitem{rubinov2013abstract}
Alexander~M. Rubinov.
\newblock {\em Abstract convexity and global optimization}, volume~44.
\newblock Springer Science \& Business Media, 2013.

\bibitem{themelis2020douglas}
Andreas Themelis and Panagiotis Patrinos.
\newblock {D}ouglas--{R}achford splitting and {ADMM} for nonconvex
  optimization: Tight convergence results.
\newblock {\em SIAM Journal on Optimization}, 30(1):149--181, 2020.

\bibitem{wang2022bregman}
Xianfu Wang and Heinz~H. Bauschke.
\newblock The {B}regman proximal average.
\newblock {\em SIAM Journal on Optimization}, 32(2):1379--1401, 2022.

\bibitem{wang2024mirror}
Ziyuan Wang, Andreas Themelis, Hongjia Ou, and Xianfu Wang.
\newblock A mirror inertial forward-reflected-backward splitting: Convergence
  analysis beyond convexity and {L}ipschitz smoothness.
\newblock {\em Journal of Optimization Theory and Applications}, 2024.

\end{thebibliography}

\end{document}